\documentclass[10pt,oneside]{amsart}
\usepackage{amscd,amssymb, bbm,color}
\usepackage{graphicx, amsmath, amssymb, amscd, amsthm, euscript, psfrag, amsfonts,bm}

\usepackage{amscd,amssymb}

\usepackage[all,arc]{xy}
\usepackage{enumerate}
\usepackage{mathrsfs}
\usepackage{amsmath}
\usepackage{graphicx}
\usepackage{caption}
\usepackage{subcaption}

\newtheorem{thm}{Theorem}[section]
\newtheorem{Thm}[thm]{Theorem}

\newtheorem{cor}[thm]{Corollary}
\newtheorem{prop}[thm]{Proposition}
\newtheorem{lem}[thm]{Lemma}

\theoremstyle{definition}
\newtheorem{defn}[thm]{Definition}\newtheorem{Def}[thm]{Definition}\newtheorem{dfn}[thm]{Definition}

\theoremstyle{remark}
\newtheorem{Rmk}[thm]{Remark}
\newtheorem{remark}[thm]{Remark}
\newtheorem{rmk}[thm]{Remark}

\newcommand{\PSL}{\operatorname{PSL}_2(\mathbb{R})}

\newcommand{\be}{begin{equation}}

\newcommand{\bH}{\mathbb H}

\newcommand{\e}{{\epsilon}}
\newcommand{\z}{\mathbb{Z}}

\newcommand{\N}{\mathbb{N}}
\renewcommand{\c}{\mathbb{C}}
\newcommand{\br}{\mathbb{R}}

\newcommand{{\grinv}}{{\Cal G}^{-r}}

\newcommand{\ba}{\backslash}

\newcommand{\G}{\Gamma}

\newcommand{\Haar}{\operatorname{Haar}}

\newcommand{\Cal}{\mathcal}
\renewcommand{\P}{\mathcal P}

\newcommand{\la}{\langle}
\newcommand{\ra}{\rangle}

\newcommand{\bp}{\begin{pmatrix}}
\newcommand{\ep}{\end{pmatrix}}

\renewcommand{\bp}{{\rm bp}}

\newcommand{\SO}{\operatorname{SO}}

\newcommand{\T}{\operatorname{T}}

\newcommand{\bT}{\mathbf T}

\newcommand{\Q}{\mathsf{Q}}

\newcommand{\PS}{\rm{PS}}

\newcommand{\op}{\operatorname}

\newcommand{\BR}{\operatorname{BR}}

\renewcommand{\setminus}{-}

\renewcommand{\be}{\begin{equation}}
\newcommand{\ee}{\end{equation}}

\newcommand{\Z}{\z}

\newcommand{\Pp}{\hat \sigma}

\newcommand{\BMS}{\op{BMS}}
\newcommand{\cA}{\mathsf A}
\renewcommand{\T}{\op{T}}

\newcommand{\eg}{\ell(\gamma)}
\renewcommand{\P}{\mathsf P}
\newcommand{\PP}{\mathcal P}

\makeatletter
\let\c@equation\c@thm
\makeatother
\numberwithin{equation}{section}
\title[Local Mixing on abelian covers] {Local Mixing and invariant measures
for horospherical subgroups on abelian covers}

\author{Hee Oh}
\address{Mathematics department, Yale university, New Haven, CT 06511 and Korea Institute for Advanced Study, Seoul, Korea}
\email{hee.oh@yale.edu}

\thanks{Oh was supported in part by NSF Grant \#1361673.}

\author{Wenyu Pan}
\address{Mathematics department, Yale university, New Haven, CT 06511}
\email{wenyu.pan@yale.edu}

\begin{document}

\maketitle

\begin{abstract} Abelian covers of hyperbolic $3$-manifolds are ubiquitous.
We prove the local mixing theorem of the frame flow for abelian covers of closed hyperbolic $3$-manifolds.
We obtain a  classification theorem for measures invariant under the horospherical subgroup. We also describe
applications to the prime geodesic theorem as well as to other counting and equidistribution problems.
Our results are proved for any abelian cover of a homogeneous space $\Gamma_0\ba G$ 
where $G$ is a rank one simple Lie group and $\Gamma_0<G$ is a convex cocompact Zariski dense subgroup.
 \end{abstract}

\section{Introduction}
\subsection{Motivation} Let $\mathcal M$ be a closed hyperbolic $3$-manifold. We can present $\mathcal M$ as the quotient
$\Gamma_0 \ba \bH^3$ of the hyperbolic $3$-space $\bH^3$ for some co-compact lattice $\Gamma_0$ of
$G=\op{PSL}_2(\c)$.
The frame bundle $\mathcal F(\mathcal M)$ is isomorphic to the homogeneous space $\Gamma_0\ba G$
and the frame flow on  $\mathcal F(\mathcal M)$ corresponds to the right multiplication
of $a_t=\begin{pmatrix} e^{t/2} & 0\\ 0 &e^{-t/2}\end{pmatrix}$ on $\Gamma_0\ba G$.

The strong mixing property of the frame flow \cite{HM} is well-known:
\begin{thm}[Strong mixing] For any $\psi_1, \psi_2\in L^2(\Gamma_0\ba G)$,
$$\lim_{t\to \infty} \int \psi_1(xa_t)\psi_2(x) \, dx =\int \psi_1 dx \cdot \int \psi_2\, dx$$
where $dx$ denotes the $G$-invariant probability measure on $\Gamma_0\ba G$.
\end{thm}

This mixing theorem and its effective refinements are of fundamental importance in homogeneous dynamics, and have
many applications  in various problems in geometry and number theory.

One recent spectacular application  was found in the resolution of the surface subgroup conjecture by Kahn-Markovic \cite{KM}. 
Based on their work, as well as Wise's, Agol settled the virtually infinite betti number conjecture \cite{Ag}:
 
 \begin{thm} Any closed hyperbolic $3$-manifold $\mathcal M=\Gamma_0\ba \bH^3$ {\it virtually} has a $\z^d$-cover for any $d\ge 1$.
 \end{thm}

That is, after passing to a subgroup of finite index,
$\Gamma_0$ contains a normal subgroup $\Gamma$ with $\Gamma\ba \Gamma_0\simeq \z^d$, and
hence the hyperbolic $3$-manifold $\Gamma\ba \bH^3$ is a regular cover of $\mathcal M$ whose deck transformation group is isomorphic to $\z^d$.

 On the frame bundle of the $\z^d$-cover $\Gamma\ba \bH^3$, the strong mixing property fails  \cite{HM}, because
for any $\psi_1, \psi_2\in C_c(\Gamma\ba G)$,
 $$\lim_{t\to \infty} \int_{\Gamma\ba G} \psi_1(xa_t)\psi_2(x) \, dx =0.$$

The aim of this paper is to formulate and prove the {\it local mixing} property of the frame flow for any abelian
cover of a closed hyperbolic $3$-manifold, or more generally for any abelian cover
of a convex cocompact rank one locally symmetric space. The local mixing property is an appropriate substitute of the strong mixing property in the setting of an infinite volume homogeneous space, and
has similar applications. We will establish a  classification theorem for measures invariant under the horospherical subgroup, extending the work
of Babillot, Ledrappier and Sarig. We will also describe
applications to the prime geodesic theorem as well as other counting and equidistribution problems.

\subsection{ Local limit theorem.} In order to motivate our definition of the local mixing property, we
recall the classical local limit theorem on the Euclidean space $\br^d$ \cite{Bre}:
\begin{thm}[Local limit theorem]  
For any absolutely continuous compactly supported probability measure $\mu$ on $\br^d$,
and any continuous function $\psi$ on $\br^d$ with compact support,
$$\lim_{n\to +\infty} n^{d/2} \int \psi \; d\mu^{*n}=c(\mu) \int_{\br^d} \psi(x) dx$$
where $\mu^{*n}$ denotes the $n$-th convolution of $\mu$ and $c(\mu)>0$ is a constant depending only on
 $\mu$. \end{thm}
The virtue of this theorem is that although  the sequence $\mu^{*n}$ weakly converges to zero, the re-normalized measure
$n^{d/2}\mu^{*n}$ converges to a non-trivial locally finite measure on $\br^d$, which is the Lebesgue measure
in this case.

\subsection{ Local mixing theorem.}  
Let $G$ be a connected semisimple linear Lie group and $\Gamma<G$ a discrete subgroup. Let $\{a_t: t\in \br \}$ be a one-parameter
diagonalizable subgroup of $G$, acting on $\Gamma\ba G$ by right translations. 
Denote by  $C_c(\Gamma\ba G)$ the space of all continuous functions with compact support.
For  a compactly supported probability measure $\mu$  on $\Gamma \ba G$, 
we consider the following family $\{\mu_t\}$ of probability measures on $\Gamma\ba G$ translated by the flow $a_t$: for $\psi\in C_c(\Gamma\ba G)$,
$$\mu_t (\psi):=\int_{\Gamma\ba G}\psi(xa_t) d\mu(x).$$

We formulate the following notion, which is analogous to the local limit theorem for $\br^d$:
\begin{dfn}   \rm A probability measure $\mu$ on $\Gamma\ba G$ has the local mixing property for $\{a_t\}$
if
there exist a positive function $\alpha$ on $\br_{>0}$ and a non-trivial Radon measure $\mathsf{m}$ on $\Gamma \ba G$
such that for any $\psi\in C_c(\Gamma\ba G)$,
$$\lim_{t\to +\infty} \alpha(t) \int \psi (x)\, d\mu_t(x) =\int \psi \,d\mathsf{m}(x).$$
\end{dfn}

If $\Gamma<G$ is a lattice, then any absolutely continuous probability measure
$\mu$ has the local mixing property for $\{a_t\}$ \cite{HM}. If $\Gamma=\{e\}$, no probability measure has
the local mixing for $\{a_t\}$.

\medskip

We focus on the rank one situation in which case
the action of $a_t$ induces
 the geodesic flow on the corresponding locally symmetric space.
Throughout the introduction, suppose that $G$ is a connected simple Lie group of real rank one, that is, $G$ is the group of orientation preserving isometries
of a simply connected Riemmanian space $\tilde X$ of rank one.
Let $\Gamma_0<G$ be a Zariski dense and  convex cocompact subgroup of $G$, i.e.
its convex core, which is the quotient of
 the convex hull of the limit set of $\Gamma_0$ by $\Gamma_0$, is compact. For instance,
 $\Gamma_0$ can be a cocompact lattice of $G$. 
 Let $$\Gamma< \Gamma_0$$ be a normal subgroup with $\z^d$-quotient.
Then $X:=\Gamma\ba \tilde X$ is a regular cover of $X_0:=\Gamma_0\ba \tilde X$ whose group of deck transformations
is isomorphic to $\z^d$.
  Let $\{a_t\in G:t\in \br\}$ be a one-parameter subgroup which is the
  lift of the geodesic flow  $\mathcal G^t$ on the unit tangent bundle $\T^1(X)$.
   When $\tilde X$ is a real hyperbolic space,
  the $a_t$ action on $\Gamma\ba G$ corresponds to the frame flow on the oriented frame bundle of $X$.

\begin{figure}\label{closing}
  \begin{center}
   \includegraphics[width=2.0in]{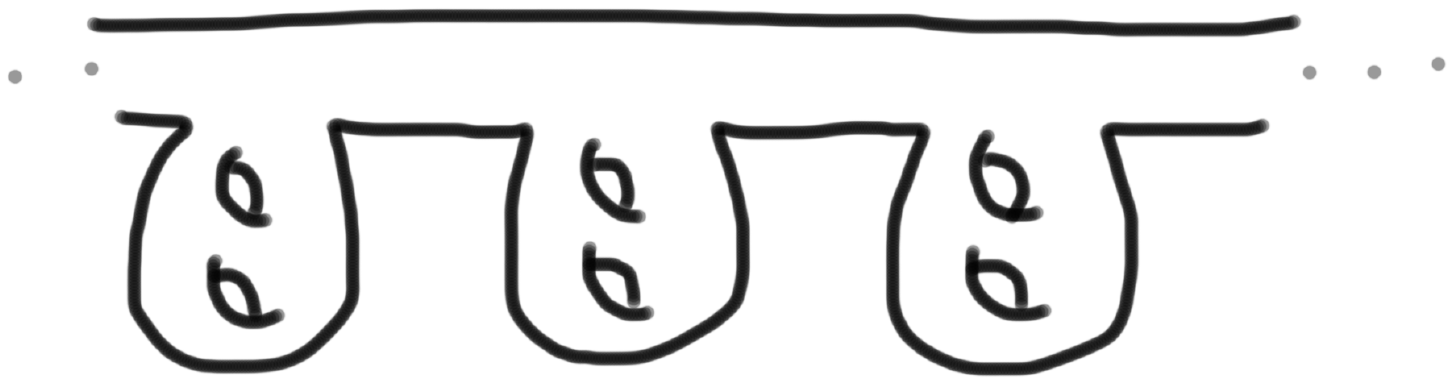} 
   \includegraphics[width=2.0in]{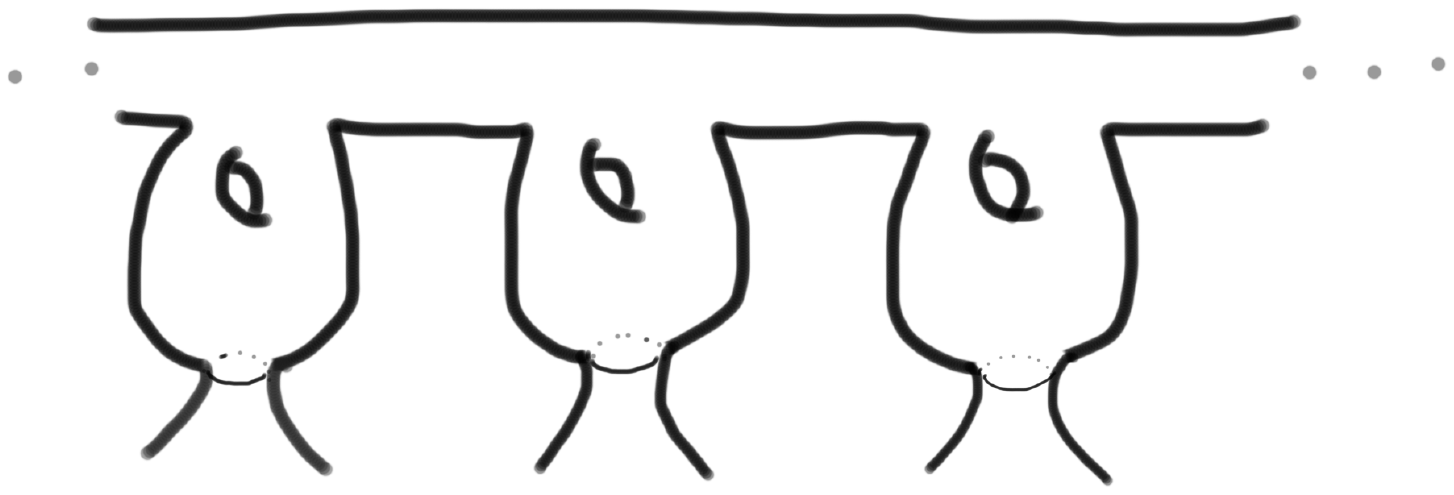}
 \end{center}\caption{$Z$-covers of convex cocompact surfaces}
\end{figure}

Denote by $\mathcal P_{\mathsf{acc}}(\Gamma\ba G)$
the space of compactly supported absolutely continuous
 probability measures on $\Gamma\ba G$ with continuous densities.

\begin{thm} [Local mixing theorem] \label{mm1} 
Any $\mu\in \mathcal P_{\mathsf{acc}}(\Gamma\ba G)$ has the local mixing property for $\{a_t\}$: for $\psi\in C_c(\Gamma \ba G)$,
$$\lim_{t\to +\infty} t^{d/2} e^{(D-\delta) t} \int \psi \, d\mu_t=c(\mu) \cdot  \int \psi \,d\mathsf{m}^{\BR_+}$$
where  $D$ is the volume entropy of  $\tilde X$, $\delta$ is the
critical exponent of $\Gamma_0$, $\mathsf{m}^{\BR_+}$ is the Burger-Roblin measure on $\Gamma\ba G$
for the expanding horospherical subgroup,
and
$c(\mu)>0$ is a constant depending on $\mu$.
\end{thm}
 We remark that  the
critical exponent $\delta$ of $\Gamma_0$ is same as that of $\Gamma$ by \cite{DS}. 

 Denote by $dx$ a $G$-invariant measure on $\Gamma\ba G$.
  Theorem \ref{mm1} is deduced from the following theorem, which describes the precise asymptotic
of the correlation functions.
 
\begin{thm} \label{mix2}
 For  $\psi_1, \psi_2\in C_c(\Gamma \ba G)$,
$$ \lim_{t\to +\infty} {t^{d/2} e^{(D-\delta)t}} \int_{\Gamma \ba G}\psi_1(xa_t) \psi_2(x) \; dx= 
 \frac{\mathsf{m}^{\op{BR_+}}( \psi_1) \mathsf{m}^{\op{BR}_-}( \psi_2)}{(2\pi\sigma)^{d/2} \mathsf m^{\BMS}(\Gamma_0\ba G)},
$$
where $\mathsf{m}^{\BMS}$ is the Bowen-Margulis-Sullivan measure on $\Gamma_0\ba G$,
$\mathsf{m}^{\BR_-}$ is the Burger-Roblin measure on $\Gamma\ba G$
for the contracting horospherical subgroup,
and $\sigma=\sigma_\G>0$ is
a constant given in \eqref{ds}.
\end{thm}

For the trivial cover, i.e. when $d=0$, this theorem was obtained  by Winter \cite{Win}, based on the earlier
work of Babillot \cite{Ba}.

If $\Gamma_0<G$ is cocompact, then $D=\delta$ and the measures $\mathsf m^{\BR_+}$,
$\mathsf m^{\BR_-}$ and $\mathsf m^{\BMS}$ are all proportional to the invariant measure $dx$. Hence
 the following is a special case of Theorem \ref{mix2}:
\begin{thm} \label{mix123} Let $\Gamma\ba G$ be a $\z^d$-cover of a compact rank one space $\Gamma_0\ba G$.
 For $\psi_1, \psi_2\in C_c(\Gamma \ba G)$, we have
$$ \lim_{t\to +\infty} {t^{d/2} } \int \psi_1(xa_t) \psi_2(x) \; dx= 
\frac{1}{(2\pi\sigma)^{d/2} }\int \psi_1 dx\int \psi_2 dx .
$$
\end{thm}

\begin {rmk}
\begin{enumerate}
\item Let $S_g$ be a compact hyperbolic surface of genus $g$ and $\Gamma_0<\PSL$ be a realization of
the surface group $\pi_1(S_g)$.
Then $\G:=[\Gamma_0, \Gamma_0]$ is a normal subgroup of $\Gamma_0$ with 
 $\z^{2g}$-quotient, and $\G\ba \bH^2$ is the homology cover of $S_g$. 
Theorem \ref{mix123} is  already new in this case.
\item 
For any $n\ge 2$, there is a congruence lattice
of $\SO(n,1)$ admitting co-abelian subgroups of infinite index  (\cite{Mi}, \cite{LM}, \cite{Lu}).
 Moreover, such a congruence lattice can be found
in any arithmetic subgroup of $\SO(n,1)$ if $n\ne 3,7$. 

\item  An infinite abelian cover of a compact quotient $\Gamma_0\ba G$
 may exist only when $G$ is $\op{SO}(n,1)$ or $\op{SU}(n,1)$, since  
other rank one groups have Kazhdan's property $T$, which forces the vanishing  of the first Betti number of any lattice in $G$.
 On the other hand, there are plenty of 
normal subgroups of a convex cocompact subgroup $\G_0$ of any $G$ with $\z^d$-quotients; for instance,
if $\Gamma_0$ is a Schottky group generated by $g$-elements, we have a $\z^{d}$-cover of $\Gamma_0\ba G$ for any $1
\le d \le g$.

\item All of our results can be generalized to any co-abelian subgroup $\Gamma$ of $\Gamma_0$
as there exists a co-finite subgroup $\Gamma_1<\Gamma_0$, which is necessarily convex cocompact
and such that $\Gamma\ba \Gamma_1$ is isomorphic to $\z^d$ for some $d\ge 0$.
 \end{enumerate}
\end{rmk}

\begin{rmk} \rm We expect Theorem \ref{mm1} to hold in a greater generality where the quotient
group $N:=\Gamma\ba \Gamma_0$ is a finitely generated nilpotent group, and the exponent $d$ 
is the polynomial growth of $N$ \cite{Bas}: for a finite generating set $S$ of $N$, there is $\beta>1$ such that
$\beta^{-1} n^d \le  \#S^n \le \beta n^d$ for all $n\ge 1$. 
\end{rmk}

\subsection{Ergodicity of the frame flow on abelian covers.}
We also establish the following ergodic property of the $A:=\{a_t\}$-action on $\Gamma\ba G$.

\begin{thm} \label{Ree}The  Bowen-Margulis-Sullivan measure $\mathsf m^{\BMS}$  on $\Gamma\ba G$
 is ergodic for the $A$-action if and only if $d\le 2$.
\end{thm}

This strengthens the previous works of Rees \cite{Re} and Yue 
\cite{Yu} on the equivalence of the 
ergodicity of geodesic flow and the condition $d\le 2$

\subsection{Measure classification for a horospherical subgroup action.}
Let  $N$ denote the contracting horospherical
subgroup of $G$:
$$N=\{g\in G: a_{-t} ga_t\to e\text{ as $t\to +\infty$}\}.$$ 
Let $M$ be the compact subgroup which is the centralizer of  $A$, so that $G/M$
is isomorphic to the unit tangent bundle $\T^1(\tilde X)$.

Let $(\Gamma\ba \Gamma_0)^*$ denote the group of characters of the abelian group $\Gamma\ba \Gamma_0\simeq \z^d$.
 Babillot and Ledrappier constructed a family $\{\mathsf m_\chi: \chi \in (\Gamma\ba \Gamma_0)^*\}$
of $NM$-invariant Radon measures on $\Gamma \ba G$ (\cite{Ba1}, \cite{BL}); see Def. \ref{bab} for a precise definition.
Each $\mathsf m_\chi$ is known to be $NM$-ergodic (\cite{BL}, \cite{Po2}, \cite{Co}).

 We show that $\mathsf m_\chi$ is $N$-ergodic and deduce the following theorem from the classification
 result of $NM$-ergodic invariant measures due to
 Sarig  and Ledrappier  (\cite{Sa}, \cite{Le}):

 \begin{thm} \label{horo}Let $\Gamma_0<G$ be cocompact and $\Gamma\ba \Gamma_0\simeq \z^d$.
 Any $N$-invariant ergodic Radon measure on $\Gamma \ba G$
is proportional to $\mathsf m_\chi$ for some $\chi \in (\Gamma\ba \Gamma_0)^*$.
\end{thm}

\subsection{  Prime geodesic theorems and holonomies.}
For $T>0$, let $\PP_T$
be the collection of all primitive closed geodesics in $\T^1(X)$  of length at most $T$.
To each closed geodesic $C$,
we can associate a conjugacy class $h_C$ in $M$, called  the holonomy class of $C$.

We normalize $\mathsf{m}^{\BMS}$ so that $\mathsf{m}^{\BMS}(\Gamma_0\ba G)=1$.
We write $f(T)\sim g(T)$ if $\lim_{T\to \infty} f(T)/g(T)=1$.
\begin{thm}\label{mp1} 
Let $\Omega\subset \T^1(X)$ be a compact subset with 
BMS-negligible boundary and  $\xi \in C(M)$  a class function.
Then as $T\to \infty$,
$$\sum_{C\in \PP_T}\frac{\ell (C\cap \Omega)}{\ell (C)}\xi (h_C  ) \sim \frac{e^{\delta T}}{(2\pi \sigma)^{d/2}\delta T^{d/2+1}} \mathsf{m}^{\BMS}(\Omega) \int_M \xi d{m} $$
where  $d{m}$
is the probability Haar measure on $M$.
\end{thm}

For $d=0$, the above theorem was proved earlier in  \cite{MMO} (also \cite{SW} for $\Gamma$ lattice). Indeed, 
given Theorem \ref{mix2}, 
the proof of \cite{MMO} applies in the same way.
Another formulation of the prime geodesic theorem in this setting
would be studying the distribution of closed geodesics  in $\T^1(X_0)$ satisfying 
some homological constraints. An explicit main term in this setting was first described by Phillips and Sarnak \cite{PS2}.
See also (\cite{KS1}, \cite{KS2}, \cite{La2}, \cite{PP}, \cite{An}) for subsequent works.

\subsection{ Distribution of a discrete $\G$-orbit in $H\ba G$.}
Let $H$ be either the trivial, a horospherical or  a symmetric subgroup of $G$  (that is, $H$ is the group
of  fixed points
 of an involution of $G$).
Another application of the local mixing result can be found
 in the study of the distribution of a discrete $\Gamma$-orbit on
 the quotient space $H\ba G$.  
  That this question can be approached by
  a mixing type result has been understood first by Margulis \cite{Mar} at least when $H$ is compact. It was  further developed by
  Duke, Rudnick and Sarnak \cite{DRS}, and Eskin and McMullen \cite{EM} when $\Gamma$ is a lattice.  See \cite{Ro}, \cite{OS}, \cite{MO}, \cite{Oh} for generalizations to geometrically finite
  groups $\Gamma$. The following theorem extends especially the works of \cite{OS} and \cite{MO}
to geometrically infinite groups which are co-abelian subgroups of convex cocompact groups.
 We give an explicit formula for a
 Borel measure $\mathcal M=\mathcal M_\Gamma $ (see Def. \ref{MD}) on $H\ba G$ for which the following holds:
 
  \begin{thm} \label{mcc} 
Suppose that $[e]\Gamma_0\subset H\ba G$ is discrete and that
$[H\cap \Gamma_0: H\cap \Gamma]<\infty$. If $B_T$ is a well-rounded sequence of compact subsets in $H\ba G$ with respect to 
 $\mathcal M$ (see Def. \ref{MDW}),
 then
 $$\# [e]\Gamma \cap B_T \sim \mathcal M(B_T) .$$
  \end{thm}
In the case where $H$ is compact, $\Gamma_0$ is cocompact, and 
 $B_T$ is the Riemannian ball in $G/K=\tilde X$, 
 this result implies that for any $o\in \tilde X$,
 $$\#\{\gamma (o)\in \Gamma (o): d(\gamma(o), o) <T\} \sim c
 \frac{e^{D T}}{T^{d/2}} .$$ This was earlier obtained by Pollitcott and Sharp
 \cite{PS} (also see \cite{Ep}). The novelty  of Theorem \ref{mcc} lies in the treatment 
 the homogeneous space $H\ba G$ with non-compact $H$ and of sequence $B_T$ of very general shape
 (e.g. sectors).

We present a concrete example:
let $Q=Q(x_1, \cdots, x_{n+1})$ be a quadratic form of signature $(n,1)$ for $n\ge 2$, and let 
$G=\op{SO}(Q)$ be the special orthogonal group
preserving $Q$.  Let $\G$ be a normal subgroup of a Zariski dense convex cocompact subgroup $\Gamma_0$ with $\z^d$-quotient.

 \begin{cor}
 Let $w_0\in \br^{n+1}$ be a non-zero vector such that $w_0\Gamma_0$ is discrete, and 
 $[\op{Stab}_{\Gamma_0}(w_0):\op{Stab}_{\Gamma}(w_0)]<\infty$. Then for any norm $\| \cdot \|$ on $\br^{n+1}$,
 $$ \#\{ v\in w_0\Gamma  : \|v\|\le T \} \sim c \frac{T^{\delta }}{(\log T)^{d/2}}$$
where $c>0$ depends only on $\Gamma$ and $\|\cdot\|$.
 \end{cor}
 In this case, the $G$-orbit $w_0G$ is isomorphic to $H\ba G$ where $H$ is either $\op{SO}(n-1,1)$, $\op{SO}(n)$
 or $MN$
  according as $Q(w)>0$, $Q(w)<0$, or $Q(w)=0$. Under this isomorphism, the norm balls give rise to a well-rounded family of compact subsets, say $B_T$
  and the explicit computation of the $\mathcal M$-measure of $B_T\subset H\ba G$ gives the above asymptotic.

 \subsection{ Discussion of the proof.}
 The proof of Theorem \ref{mix2} is based on extending
the symbolic dynamics approach of studying the geodesic flow on $\T^1(X_0)$ as the suspension flow
on $\Sigma\times \br/\sim$ for a subshift of finite type $(\Sigma,\sigma)$.
The $a_t$ flow on a $\z^d$-cover $\Gamma\ba G$ can be
studied via the suspension flow on $$\tilde \Sigma:=\Sigma \times\z^d   \times M\times \br/\sim$$
where the equivalence is defined via the shift map $\sigma:\Sigma\to \Sigma$, the first return time map $\tau:\Sigma\to \br$,
the $\z^d$-coordinate map $f:\Sigma\to \z^d$ and the holonomy map $\theta:\Sigma\to M$.
The asymptotic behavior of the correlation function of the suspension flow on $\tilde \Sigma$ with respect to the BMS measure
can then be investigated using analytic properties of the associated Ruelle transfer
operators $L_{s, v, \mu}$  of three parameters $s\in \c, v\in \hat{\z}^d, \mu\in \hat M$ 
where $\hat{\z}^d$ and $\hat M$ denote the unitary dual of $\z^d$ and of $M$ respectively (see Def. \ref{ruelle}).
The key ingredient is to show that on the plane $\Re(s)\ge \delta$, 
the map $$s\mapsto (1-L_{s,v,\mu})^{-1}$$
 is holomorphic
except for a simple pole at $s=\delta$, which occurs only when both $v$ and $\mu$
are trivial. To each element $\gamma\in \Gamma_0$ we can associate the length  $\ell(\gamma)\in \br$ and
the Frobenius element
$f(\gamma)\in \z^d$ 
 and the holonomy representation $\theta(\gamma)\in M$.
Our proof of the desired analytic properties of $(1-L_{s,v,\mu})^{-1}$ is based on the study
of the generalized length spectrum of $\Gamma_0$ relative to $\Gamma$:
$$\mathcal GL(\Gamma_0,\Gamma):= \{(\ell(\gamma), f(\gamma), \theta(\gamma))\in \br \times \z^d\times M:\gamma\in \G_0\} .$$
The correlation function for the BMS measure can then be expressed in terms of  the operator
$(1-L_{s,v,\mu})^{-1}$ via an appropriate Laplace/Fourier transform.
We then perform the necessary Fourier analysis to extract the main term coming from the residue.
Finally we can deduce the precise asymptotic of the correlation function for the Haar measure
from that for the BMS measure using ideas originated in Roblin's work (see Theorem \ref{com}). 
In order to prove Theorem \ref{horo}, we first deduce from Theorem \ref{mix2} and the closing lemma that 
the group generated by
the generalized length spectrum
$\mathcal GL(\Gamma_0,\Gamma) $ is dense in $\br \times \z^d\times M$. Using this, we show that 
for any {\it generalized} BMS measure $\mathsf m_\Gamma$ on $\Gamma \ba G$,
 any $N$-invariant measurable function on $\Gamma\ba G$ is invariant by $\Gamma\ba\Gamma_0$ on the left and
 by $AM$ on the right. Then the $AM$-ergodicity of $\mathsf m_\Gamma$, considered as a measure on $\Gamma_0\ba G$,
 implies Theorem \ref{horo}, since the transverse measure of a Babillot-Ledrappier measure
 $\mathsf{m}_{\chi}$ equals to the transverse measure of some generalized BMS measure.
\medskip
 
 \subsection{Organization.}
The paper is organized as follows. In section 2, we introduce the suspension model for the frame flow on abelian covers. In section 3, we 
investigate the analytic properties of the Ruelle transfer operators $L_{s,v,\mu}$. 
In section 4, we deduce the
asymptotic behavior of the correlation functions of the suspension flow with respect to the BMS measure from the study
of the transfer operators made in section 3, and prove Theorem \ref{mix2}.
  In section 5, we study the ergodicity of the frame flow with respect to generalized BMS-measures.
In section 6, we discuss the ergodicity of the Babillot-Ledrappier measures for the horospherical subgroup action, and
deduce a measure classification invariant under the horospherical subgroup.
 Applications to the prime geodesic theorem, and to other counting theorems \ref{mp1} and  \ref{mcc} are discussed in the final section 7.

\medskip
\noindent{\bf Acknowlegment.} We would like to thank Curt McMullen for useful comments.

\section{Suspension model for the frame flow on abelian covers}

\subsection{Set-up and Notations.} Unless mentioned otherwise, we use the notation and assumptions made in
this section throughout the paper.
Let $G$ be a connected simple linear Lie group of real rank one and $K$ a maximal compact subgroup
of $G$. Let $\tilde X=G/K$ be the associated simply connected Riemmanian symmetric space and let
$\partial_\infty(\tilde X)$ be its geometric boundary. Choosing a unit tangent vector $v_o$ at $o:=[K]\in G/K$, the unit tangent bundle $\T^1(\tilde X)$ 
can be identified with $G/M$ where $M$ is the stabilizer subgroup of $v_o$. We let $d$ denote the right $K$-invariant and left
$G$-invariant distance function on $G$ which induces the Riemannian metric on $\tilde X$ which we will also denote by $d$.
Let $A=\{a_t\}< G$  be the one-parameter subgroup of semisimple elements whose right translation action on 
$G/M$ gives the geodesic flow. Then $M$ equals to the centralizer of $A$ in $K$.
We denote by $N^{+}$ and $N^-$ the expanding and the contracting horospherical subgroups of $G$ for the action of $a_t$:
$$N^{\pm}=\{g\in G: a_t g a_t^{-1}\to e\text{ as $t\to \pm \infty$}\}.$$
We denote by $D>0$ the volume entropy
of $\tilde X$, i.e.
$$D:=\lim_{T\to \infty} \frac{\log \op{Vol}(B(o,T))}{T}$$
where $B(o,T)=\{x\in \tilde X: d(o, x)\le T\}$.
For instance, if $\tilde X=\bH^n$, then $D=n-1$.

For a discrete subgroup $\Gamma_0<G$,
we denote by $\Lambda(\G_0)$ the limit set of $\Gamma_0$, which  is the set of all accumulation points
in $\tilde X\cup \partial(\tilde X)$ of an orbit of $\Gamma_0$ in $\tilde X$.

Let $\G_0$ be a Zariski dense and convex cocompact
 subgroup of $G$; this means that the convex hull of $\Lambda(\Gamma)$ is compact modulo $\Gamma$.

 Let $\Gamma <\Gamma_0$ be a normal subgroup with $$\Gamma\ba \Gamma_0\simeq \z^d.$$
 Set $X_0=\Gamma_0\ba \tilde X$ and $X=\Gamma\ba \tilde X$.
 So we may identify $\T^1(X_0)=\Gamma_0\ba G/M$ and $\T^1(X)=\Gamma\ba G/M$.
 The critical exponents of $\Gamma$ and $\Gamma_0$ coincide \cite{DS},
 which we will denote by  $\delta$.
 Then $0<\delta \le D$,
 and $\delta=D$ if and only if $\G_0$ is co-compact in $G$ by Sullivan \cite{Su}.
    As $\Gamma$ is normal, we have $\Lambda(\Gamma)=\Lambda(\Gamma_0)$.

  We recall the construction of the Bowen-Margulis-Sullivan measure $\mathsf{m}^{\BMS}$
  and Burger-Roblin measures $\mathsf{m}^{\BR_{\pm}}$ on $\Gamma_0\ba G$.

We let $\{\mathsf{m}_x: x\in \tilde X\}$ and $\{\nu_x: x\in \tilde X\}$ be $\G_0$-invariant conformal densities of
dimensions $D$ and $\delta$ respectively, unique up to scalings. They are called the Lebesgue density
and the Patterson-Sullivan density, respectively.

The notation $\beta_{\xi}(x,y)$ denotes the Busemann function for $\xi\in \partial(\tilde X)$,
and $x,y\in \tilde X$. 
The Hopf parametrization of $\T^1(\tilde X)$
 as $(\partial^2 (\tilde X) \setminus \text{Diagonal})\times \br$ is given by
 $$u\mapsto (u^+, u^-, s=\beta_{u^-}(o, u))$$ where  $u^{\pm}\in \partial(\tilde X)$
 are the forward and the backward end points of the geodesic determined by $u$
 and $\beta_{u^-}(o,u)=\beta_{u^-}(o, \pi(u))$ for the canonical projection $\pi:\T^1(\tilde X)\to \tilde X$.
  
Using this parametrization, the following defines locally finite Borel measures on $\T^1(\tilde X)$:
 $$d \tilde {\mathsf{m}}^{\BMS}(u)= e^{\delta \beta_{u^+} (o,u)+\delta\beta_{u^-}(0,u)}
 {d\nu_o(u^+) d\nu_o(u^-) ds};$$
 $$d \tilde{ \mathsf{m}}^{\BR_+}(u)= e^{D \beta_{u^+} (o,u) +\delta\beta_{u^-}(0,u)}
 {d\mathsf{m}_o(u^+) d\nu_o(u^-) ds};$$
 $$d \tilde {\mathsf{m}}^{\BR_-}(u)= e^{\delta \beta_{u^+} (o,u) +D\beta_{u^-}(0,u)}
 {d\nu_o(u^+) d\mathsf{m}_o(u^-) ds};$$
 $$d \tilde {\mathsf{m}}^{\Haar}(u)= e^{D \beta_{u^+} (o,u) + D\beta_{u^-}(0,u)}
 {d\mathsf{m}_o(u^+) d\mathsf{m}_o(u^-) ds};$$
 
They are left $\Gamma_0$-invariant  measures on $\T^1(\tilde X)=G/M$. We will use the same notation for
their $M$-invariant lifts  to $G$, which are, respectively, right $AM$, $N^+M$, $N^-M$ and $G$-invariant.
By abuse of notation, the induced measures on $\Gamma_0\ba G$ will be denoted by
$\mathsf{m}^{\BMS}, \mathsf{m}^{\BR_\pm}, \mathsf{m}^{\Haar}$ respectively. 
If $\Gamma_0$ is cocompact in $G$, these measures are all equal to each other, being simply the Haar measure.
In general,
only $\mathsf{m}^{\BMS}$ is a finite measure on $\Gamma_0\ba G$.
An important feature of $\mathsf{m}^{\BMS}$ is that it is the unique measure of maximal entropy (which is $\delta$)
 as a measure on $\T^1(X_0)$.
 
 Since the measures $\tilde {\mathsf{m}}^{\BMS}, \tilde{\mathsf{m}}^{\BR_\pm}, \tilde{\mathsf{m}}^{\Haar}$ are
 all $\Gamma$-invariant as $\Gamma<\Gamma_0$, they also induce measures on $\Gamma\ba G$ for which we will
 use the same notation $\mathsf{m}^{\BMS}, \mathsf{m}^{\BR_\pm}, \mathsf{m}^{\Haar}$ respectively. 

 We will normalize $\mathsf{m}^{\BMS}$ so that
 $$\mathsf{m}^{\BMS}(\Gamma_0\ba G)=1 $$
 which can be done by rescaling $\nu_o$.

\begin{rmk} \rm  We remark that $\{\nu_x\}$ is also the unique $\Gamma$-invariant conformal density of dimension $\delta$ whose
 support is $\Lambda(\Gamma)$, up to a constant multiple; this can be deduced from  \cite[Prop. 11.10, Thm. 11.17]{PPS}). Therefore
 the BMS-measures and BR measures on $\Gamma\ba G$ can be defined canonically without the reference to $\Gamma_0$.
 \end{rmk}

\subsection{ Markov sections and suspension space $\Sigma^{\tau}$.}
Denote by $\Omega_0 $ the non-wandering set of the geodesic flow
$\{a_t\}$ in $ \Gamma_0\ba G/M$, i.e. $$\{x:\text{for any neighborhood $ U$ of $x$}, 
Ua_{t_i}\cap U\ne \emptyset \;\text{ for some $t_i\to \infty$}\}.$$
 The set $\Omega_0$ coincides with the support of $m^{\BMS}$ and is a hyperbolic set for the geodesic flow.
In this subsection, we recall the well-known construction
of a subshift of finite type and its suspension space which gives a symbolic space model
for $(\Omega_0, a_t)$ (see  \cite{PP}, \cite{Bo}, \cite{HK}, \cite{Ra}, \cite{PS} for a general reference).
For each $z\in \Omega_0$, the strong unstable manifold $W^{su}(z)$,
 the strong stable manifold $W^{ss}(z)$, the weak unstable manifold $W^u(z)$ and the weak stable manifold $W^s(z)$ 
 are respectively given by the sets
   $zN^+$, $zN^-, zN^+A, $ and $ zN^-A$ in $\Gamma_0\ba G/M$ respectively.

 Consider a finite set $z_1, \ldots , z_k$ in $\Omega_0$ and choose small compact neighborhoods $U_i$ and $S_i$ of 
 $z_i$ in $W^{su}(z_i) \cap \Omega_0$ and  $W^{ss}( z_i)\cap \Omega_0$ respectively 
 such that $U_i = \overline{\mbox{int}^u(U_i)}$ and
  $S_i = \overline{\mbox{int}^s(S_i)}$. 
Here  $\mbox{int}^u(U_i)$ denotes the interior of $U_i$ in the set $W^{su}(z_i) \cap \Omega_0$ and $\mbox{int}^s(S_i)$ is defined similarly. 
 For $x \in U_i$ and $y \in S_i$, we write $[x, y]$ for the unique local intersection of 
 $W^{ss}(x)$ and $W^{u}(y)$. We call the following sets {\it rectangles}:
\[ R_i = [U_i, S_i] := \{ [x, y]: x\in U_i, y \in S_i \}\]
and denote their interiors by
$\mbox{int}(R_i)= [\mbox{int}^u(U_i), \mbox{int}^s(S_i)] .$
Note that $U_i = [U_i, z_i] \subset R_i$.

Given a disjoint union
 $ R = \cup_i R_i$ of rectangles such that $R A=\Omega_0$, the first return time
$\tau: R \rightarrow \mathbb{R}_{>0}$  and the first return map $\mathcal{P} : R \to R$ are given by
 $$\tau(x) := \inf \lbrace t> 0 : xa_t \in R\rbrace \quad \text {and} \quad \mathcal P(x):= xa_{\tau(x)}.$$
 The associated
 transition matrix $\cA$ is the $k\times k$ matrix given by
\[ {\cA}_{lm} = \left\lbrace  \begin{array}{ll} 1 \quad
\mbox{ if }\mbox{int} (R_l) \cap \mathcal{P}^{-1} \mbox{int} (R_m)  \neq \emptyset \\  0 \quad \mbox{ otherwise.}  \end{array} \right.\]

Fix $\e>0$ much smaller than the injectivity radius of $\Gamma_0\ba G$.
By Ratner \cite{Ra} and Bowen \cite{Bo},
we have  a Markov section for the flow $a_t$ of size $\e$, that is,
a family $\mathcal{R} = \lbrace R_1 , \ldots  ,R_k \rbrace$ of disjoint rectangles satisfying the following:
\begin{enumerate}
\item $ \Omega_0 = \cup_1^k R_i a_{[0, \e]}$
\item the diameter of each $R_i$ is at most $\e$, and 
\item for any $i \neq j$, at least one of $R_i \cap R_ja_{[0, \e]}$ or $R_j \cap R_ia_{[0, \e]}$ is empty.
\item $ \mathcal{P}([\operatorname{int}^uU_i, x])  \supset [\operatorname{Int}^uU_j, \mathcal P (x)]$
   and $ \mathcal{P}([x, \operatorname{Int}^sS_i] )) \subset [\mathcal P (x),\operatorname{Int}^sS_j]  $
for any $x\in \op{int} (R_i)\cap {\mathcal P}^{-1}(\op{int}(R_j))$. 
\item $\cA$ is aperiodic, i.e.
for some $N\ge 1$, all the entries of $\cA^N$ are positive.
\end{enumerate}
 Set $R:=\cup R_i$ and $U:=\cup U_i$.

Let $\Sigma$ be the space of bi-infinite sequences $x \in \lbrace 1, \ldots, k \rbrace^\mathbb{Z}$ such that 
${\cA}_{x_l x_{l+1}} = 1$ for all $l$. We denote by $\Sigma^+$ the space of one sided  sequences 
$$\Sigma^+= \lbrace (x_i)_{ i\ge 0 } : {\cA}_{x_i, x_{i + 1}} = 1 \mbox{ for all } i \geq 0 \rbrace.$$
Non-negative coordinates of $x\in \Sigma$ will be referred
to as future coordinates of $x$.
 A function on $\Sigma$ which depends only on future coordinates can be regarded as a function on 
$\Sigma^+$.

We will write $\sigma: \Sigma \rightarrow \Sigma$ for the shift map $(\sigma x)_i = x_{i+1}$.
By abuse of notation we will also denote by $\sigma$ the shift map acting on $\Sigma^+$.

For $\beta \in (0, 1)$, we can give a metric $d_\beta$ on $\Sigma$ (resp. on $\Sigma^+$) by 
\[ d_\beta(x, x') = \beta^{\inf \lbrace |j|: x_j \neq x'_j \rbrace}.\]

\begin{dfn} [The map $\zeta:\Sigma \rightarrow R$]   \rm
Let $\hat R$ be the set of $x\in R$ such that $\mathcal{P}^{m} x \in \mbox{int} (R ) $ for all $m\in \Z$.
For $ x \in \hat R$, we obtain a sequence $\omega = \omega(x) \in \Sigma$ such that
 $\mathcal{P}^k x\in R_{\omega_k}$  for all $k \in \mathbb{Z}$. The set 
$ \hat \Sigma := \lbrace \omega(x) : x \in \hat R \rbrace $
is a residual set in $\Sigma$.  This map $x\mapsto \omega(x)$ is injective on $\hat R$
as the distinct pair of geodesics diverge from each other either in positive or negative time. We can extend $\omega^{-1}:\hat \Sigma\to \hat R$
a continuous surjective function $\zeta:\Sigma \rightarrow R$, which  intertwines $\sigma$ and $ \mathcal{P}$. 
\end{dfn}
 \begin{dfn} [The map $\zeta^+ : \Sigma ^+ \rightarrow U$] \rm
 Let $ \hat U $ be the set of $ u \in U $ such that $ \Pp^{m} u \in \mbox{int}^u(U)$ for all $m\in \N\cup\{0\}$.
Similarly to the above, we can define an injective map $\hat U\to \Sigma^+$, and then
a continuous surjection $\zeta^+:\Sigma^+\to U$. 
\end{dfn}

For $\beta $ sufficiently close to $1$, the embeddings $\zeta$ and $ \zeta^+$ are Lipschitz. We fix such a $\beta$ once and for all. The space $C_\beta(\Sigma)$
(resp. $C_\beta(\Sigma^+)$)
of $d_\beta$-Lipschitz functions on $\Sigma$ (resp. on $\Sigma^+$) is a Banach space with the usual Lipschitz norm
\be\label{Li} ||\psi ||_{d_\beta} =  \sup |\psi |+ \sup_{x \neq y} \frac{|\psi(x) - \psi(y)|}{d_\beta(x, y)}.\ee

We denote $\tau\circ \zeta \in C_{\beta}(\Sigma)$ by $\tau$ by abusing the notation.

We form the suspension 
\begin{equation*}
\Sigma^{\tau}:=\Sigma \times \mathbb{R}/(x,s)\sim (\sigma x,s-\tau(x))
\end{equation*}
and the suspension flow on $\Sigma^\tau$ is given by $[(x,s)]\mapsto [(x,t+s)]$. 
The  map $[(x,t)]\to \zeta(x) a_t$ 
 is a semi-conjugacy $\Sigma^\tau\to \Omega_0$ intertwining the suspension flow and the geodesic flow $a_t$.

For $g\in C_{\beta}(\Sigma)$, called the potential function, the pressure of $g$ is defined as
\[ \op{Pr}_\sigma(g) := \sup_\mu \left( \int g d\mu + \mbox{entropy}_\mu(\sigma) \right) \]
over all $\sigma$-invariant Borel probability measures $\mu$ on ${\Sigma}$, where $\mbox{entropy}_\mu(\sigma)$ denotes the measure theoretic entropy of $\sigma$ with respect to $\mu$.
The critical exponent $\delta$
 is the unique positive number such that $\op{Pr}(-\delta \tau)=0$. 
Let $\nu$ denote the unique equilibrium measure for $-\delta \tau $, i.e.
$\delta \int \tau d\nu = \mbox{entropy}_\nu(\sigma)$.

The BMS measure $\mathsf{m}^{\BMS}$ on $\Omega_0\simeq \Sigma^\tau$
being the unique probability measure of maximal entropy for the geodesic flow,
corresponds to the measure locally given by
$\frac{1}{\int \tau d\nu} (d\nu\times ds)$
where $ds$ is the Lebesgue measure on $\br$.

For a map $g$ on $\Sigma $  or on $\Sigma^+$ and $n\ge 1$,
we write $$g_n(x)=g(x) +g(\sigma(x)) +\cdots+ g (\sigma^{(n-1)}(x)).$$

\subsection{  $\z^d\times \br$-suspension space $\Sigma^{f,\tau}$.}
Note that $X$ is a regular $\z^d$-cover of the convex cocompact 
manifold $X_0$. 
Let $p$ denote the canonical projection map
$ \T^1(X)\to \T^1(X_0)$. Then $$\Omega_X:=p^{-1}(\Omega_0)\simeq \z^d\times \Omega_0$$ is the support 
of the BMS measure in $\T^1(X)$. 
We enumerate the group of deck transformation
  $\Gamma\ba \Gamma_0$ for the covering map
  $X\to X_0$ as $\{D_\xi: \xi \in \z^d\}$ so that $D_{\xi_1}\circ D_{\xi_2}=D_{\xi_1+\xi_2}$.
 Note $D_\xi$ acts on $\Gamma\ba G/M$ as well as on $\Gamma\ba G$.
 \begin{Def}\label{zd}
Fix a precompact and connected fundamental domain $\mathcal F\subset \Gamma\ba G/M$ 
 for the $\z^d$-action on $\T^1(X)$.
 \begin{enumerate}
 \item We define the $\z^d$-coordinate of $x\in\T^1(X)$ relative to $\mathcal F$ to be
 the unique $\xi\in \z^d$ such that  $x\in D_\xi (\mathcal F)$,
 and write $\xi(x: t)$ for the $\z^d$-coordinate of $xa_t$.

 \item  Choose a continuous section, say $\mathsf s$, from $R$ into $\mathcal F$. 
 Define $f:\Sigma \to \z^d$ as follows: for $x\in \Sigma$,
 $$ f(x)=\xi( \mathsf s\circ \zeta(x)  : \tau(x) ),$$
that is, the $\z^d$-coordinate of $\mathsf s (\zeta(x))a_{\tau(x)}$.
\end{enumerate}
\end{Def}
Note that  $f(x)$ depends only on the two coordinates
$x_0$ and $x_1$, and if $\sigma^n(x)=x$,
then $f_n(x)$ is the Frobenius element of  the closed geodesic in $\T^1(X_0)$ given by $x$.

Consider the suspension space
$$\Sigma^{f,\tau}:= \Sigma\times \mathbb{Z}^d \times \mathbb{R} / (x,  \xi,s)
 \sim (\sigma (x),\xi+f(x),s-\tau(x) )  $$
with the suspension flow $[(x,\xi, s)]\to [(x,\xi, s+t)]$.
The map
$\Sigma^{f,\tau}\to \Omega_X$
given by $$[( x,\xi, t)]\mapsto 
D_{\xi}(\mathsf s\circ \zeta( x))a_t$$ is a Lipschitz surjective map intertwining the suspension flow and the geodesic flow.

If $\sigma^n(x)=x$, then
$$(x,  \xi, s+ { \tau_n}(x))\sim (\sigma^n(x), \xi +f_n(x), s )= (x, \xi+f_n(x), s).$$
Hence
 $[(x, \xi, s)]$ gives rise to a periodic orbit if and only if $\sigma^n(x)=x$ and $f_n(x)=0$ for some $n\in \N$.

\subsection{ $\z^d\times M\times \br$-suspension space $\Sigma^{f,\theta,\tau}$.}
The homogeneous space $\Gamma \ba G$ is a principal $M$-bundle
over $\T^1(X)=\Gamma \ba G/M$. 
Now take a smooth section $\mathsf S: \mathsf s(R) \to \Gamma \ba G$ by trivializing  the bundle locally.

\begin{Def} 
Define  $ \theta: \Sigma \to M$ as follows: for $x\in \Sigma$, $\theta(x)\in M$ is
the unique element satisfying
$$(\mathsf S \circ \mathsf s\circ \zeta)(x) a_{ \tau(x)} = 
D_{f(x)}(\mathsf S\circ \mathsf s \circ \zeta) ( \sigma (x)) \theta(x)^{-1}.$$
\end{Def}

We choose the section $\mathsf S$ a bit more carefully so that the resulting holonomy map $\theta$ depends only on future coordinates:
 first trivialize the bundle over each $\mathsf s(U_j)$ and extend the trivialization to $
 \mathsf s(R_j)$ by requiring 
$( \mathsf S \circ \mathsf s)([u,s_1]) $ and $( \mathsf S \circ \mathsf s)([u,s_2])$ be forward asymptotic for all $u \in U_j$ and $s_1, s_2\in S_j.$

If $x\in \Sigma$ has period $n$,
then $\theta_n(x)^{-1}$
is in the same conjugacy class as the holonomy associated to the closed geodesic $\zeta(x)A$.

\medskip

We set $\Omega$ to be the preimage of $\Omega_X$ under the projection $\Gamma\ba G \to \Gamma\ba G/M$;
this is precisely the support of $\mathsf{m}^{\BMS}$ defined as a measure on $\Gamma\ba G$ in the previous section.

 Consider the suspension space $$\Sigma^{f,\theta,\tau}:=\Sigma \times \mathbb{Z}^d\times M \times \br/
 (x,\xi,  m, s )\sim (\sigma x,  \xi+f(x), \theta^{-1} (x) m  ,s-\tau(x))   $$
 with the suspension flow $[(x,\xi, m, s)]\mapsto [(x,\xi, m, s+t)]$.
Now the map
$\pi:\Sigma^{f,\theta ,\tau}\to \Omega$
given by $$[( x,\xi, m, t)]\mapsto 
D_{\xi}(\mathsf S \circ \mathsf s\circ \zeta ( x))m a_t$$ is 
a Lipschitz surjective map intertwining the suspension flow and the $a_t$-flow.

The BMS measure $\mathsf{m}^{\BMS}$ on $\Omega$ corresponds to the measure
locally given by the product of $\nu$ and the Haar measure on $ \Z^d\times M\times\br$:
 $$d\mathsf{m}^{\BMS}=\frac{1}{\int \tau d\nu} \pi_*(d\nu  d {\xi}  dm ds).$$

\section{Analytic properties of Ruelle operators $ L_{z, v,\mu}$}
We continue the setup and notations from section 2 for $G, \Gamma, \Gamma_0,
\Sigma, M, \tau, \theta,\delta $ etc.
In particular, $\Gamma$ is a normal
subgroup of a Zariski dense convex cocompact subgroup
$\Gamma_0$ with $\Gamma\ba \Gamma_0 \simeq \z^d$ for some $d\ge 0$.
We note that the functions $\tau:\Sigma\to \br_{>0}$ and $\theta:\Sigma\to M$ depend only on future coordinates, and hence
we may regard them as functions on $\Sigma^+$. 

For $\psi\in C_\beta(\Sigma^+)$, the Ruelle operator $L_\psi:C_\beta(\Sigma^+) \to C_\beta (\Sigma^+ )$
is defined by
$$L_\psi (g)(x)=\sum_{\sigma(y)=x} e^{-\psi(y)} g(y). $$

The Ruelle-Perron-Frobenius theorem implies the following (cf. \cite{PP}):

\begin{Thm}\label{rpf}
\label{R-P-F}
\begin{enumerate}
\item  $1$ is the unique eigenvalue of the maximum modulus of $L_{\delta \tau}$, and the corresponding eigenfunction
$h\in C_\beta(\Sigma^+)$ is positive. 

 \item The remainder of the spectrum of $L_{\delta \tau}$ is contained in a disc of radius strictly smaller than $1$.

\item There exists a unique probability measure $\rho$ on $\Sigma^+$ 
such that $L_{\delta \tau }^*(\rho)=\rho$, i.e. $\int L_{\delta \tau} \psi d\rho
=\int \psi d\rho$, and $ h d\rho= d \nu$. 
\end{enumerate}
\end{Thm}

\subsection{ Three-parameter Ruelle operators on vector-valued functions.}
Denote by $\hat M$ the unitary dual of $M$, i.e. the space of all irreducible
unitary representations $(\mu, W)$ of $M$ up to isomorphism. As $M$ is compact,
they are precisely  irreducible finite dimensional representations of $M$.
We write $\mu=1$ for the trivial  representation.
Similarly, $\hat \z^d$ denotes the unitary dual of $\z^d$.
We identify $\hat \z^d$ with $\bT^d:=(\br/(2\pi \z))^d$ via
the isomorphism $\bT^d\to \hat \z^d$ given by
$\chi_v(\xi)=e^{i\langle v,\xi \rangle}$.
In our study of the correlation function of the suspension flow
 on $\Omega=\Sigma\times   \Z^d \times M\times \br/\sim$ with respect to
the BMS measure, an understanding of the spectrum 
of {\it three} parameter Ruelle operators indexed by triples $(z, v, \mu)\in \c \times \bT^d\times \hat M$
will play a crucial role.

 \begin{dfn} \label{ruelle} \rm For each triple $(z, v, (\mu, W))\in \c \times \bT^d\times \hat M$,
  define the transfer operator 
  \be\label{rue}
  L_{z, v,\mu} : C_\beta (\Sigma^+, W)\to  C_\beta (\Sigma^+, W)\ee by
$$L_{z, v,\mu} (g)(x)=\sum_{\sigma(y)=x} e^{-z\tau (y) + i \langle  v, f(y)\rangle} \mu(\theta(y)) g(y) ,$$
where $C_\beta (\Sigma^+, W)$ denotes the Banach space of $W$-valued Lipschitz maps with Lipschitz norm
defined analogously as \eqref{Li} using a Hermitian norm on $W$.
\end{dfn} 
We write $L_{z, v}$ for $L_{z,v, 1}$ and $L_{z}$ for $L_{z, 0,1}$ for simplicity.

Denoting the center of $M$ by $Z(M)$, the following is well-known:
\begin{lem}\label{zm}
The group $Z(M)$ is at most $1$-dimensional and $$M=Z(M)[M,M].$$
\end{lem}
Hence we may identify $Z(M)$ with $M/[M,M]$. We denote by $[m]\in Z(M)=M/[M,M]$ for the projection of $m\in M$.
If $\mu$ is one-dimensional,
then $\mu$ is determined by $\mu|_{Z(M)}$.
If $Z(M)$ is non-trivial,  then $Z(M)\simeq \br/(2\pi \z)$, which we may identify with $[0,2\pi)$, and hence
any one dimensional unitary representation $\mu$ is of the form
 $\chi_p(m )=e^{i p [m]}$ for some integer $p\in \z$.
In this case, we write $L_{z,v, p}$ for $L_{z,v,\mu}$.
 
\subsection{ Spectrum of Ruelle operators.} The aim of this subsection is to prove
Theorem \ref{mrr} on analytic properties of $L_{\delta+it, v, \mu}$.
We denote by $\op{Fix}(\sigma^n)$ the set of $y\in \Sigma^+$ fixed by $\sigma^n$.
The following proposition is a key ingredient in understanding the spectrum of $L_{\delta+ it, v,\mu}$'s.
 \begin{prop}\label{pmc}
 \label{den}
 \begin{enumerate}
\item There exists  $y\in \Sigma^+$ such that
 $\{(\sigma^n(y), \theta_n(y))\in \Sigma^+ \times M: n\in \mathbb N\}$
 is dense in $\Sigma^+ \times M$.

 \item 
  There exists $y\in \op{Fix}(\sigma^n)$ for some $n$ with $f_n(y)=0$ such that
  $[\theta_n(y)]$ generates a dense subgroup in $Z(M)$.
 \end{enumerate}
 \end{prop}
\begin{proof} 
Claim (1) follows
from the existence of a dense $A^+$ orbit in $\Omega_0\subset \T^1(X_0)$
which is a consequence of the $A$-ergodicity of $\mathsf{m}^{\BMS}$ on $\Omega_0$ \cite{Win} (cf. Appendix of this paper).
The claim (2) is non-trivial only when $Z(M)$ is non-trivial; in this case, $Z(M)=\op{SO}(2)$ by Lemma \ref{zm}.
Applying the work of Prasad and Rapinchuk {\cite{PR}} to $\Gamma$, we obtain a hyperbolic element $\gamma
\in \Gamma$ that is conjugate to $a_\gamma m_\gamma\in AM$
and $m_\gamma$ generates a dense subset in $Z(M)=M/[M,M]$. The element $\gamma$
defines a closed geodesic in $\Omega_0$ which again yields an
element $y\in  \op{Fix}(\sigma^n), f_n(y)=0$ and $[\theta_n(y)]=[m_\gamma]$ for some $n$.
This implies the claim.
\end{proof}

\begin{lem}
\label{sha}
 The subgroup generated by
$\cup_{n\ge 1} \{(\tau_n(y), f_n(y))\in \br \times \z^d :y\in \op{Fix}(\sigma^n) \}$ is dense in $\br\times \z^d$.
\end{lem}

\begin{proof} Denote by $H$ the subgroup in concern.
The projection of $H$ to $\z^d$ is surjective by the construction of $f$. Therefore
it suffices to show that $H\cap (\br \times \{0\})$
is $\br$.
This follows because the length spectrum of $\Gamma$ is non-arithmetic { \cite{Ki}}.
\end{proof}

We will denote by $\sigma_0(L_{z,v, \mu})$ the spectral radius of the operator $L_{z,v, \mu}$ on $C_\beta(\Sigma^+, W)$.
\begin{prop}\label{sp1}  Let $(\mu, W) \in \hat M$, and $(t,v)\in \br \times \bT^d$.
\begin{enumerate} 
\item We have $\sigma_0(L_{\delta+it, v, \mu}) \le 1$.
\item If $\sigma_0(L_{\delta+it, v,\mu})=1$, then 
$L_{\delta+it, v,\mu}$ has a simple eigenvalue of modulus one  and $\mu$ is  $1$-dimensional.
\end{enumerate} \end{prop}
\begin{proof} 
(1) and the first part of (2) follow from Theorems 8.1 and 8.3 of \cite{PP}. 

Suppose $\sigma_0(L_{\delta+it, v,\mu})=1$. Then for some {$w\in C_\beta(\Sigma^+, W)$} and $b\in \br$,
 $$L_{\delta+it, v,\mu}w=e^{ib} w.$$ 
Using the convexity argument {(as in p.54 of \cite{PP})}, it follows that
$$e^{i (-t \cdot r(y) +\langle v, f(y)\rangle)} \mu(\theta(y)) w(y)= e^{ib} w(\sigma(y))$$
for all $y\in \Sigma^+$. In other words,
$$e^{i (-t \cdot r(y) +\langle v, f(y)\rangle -b)}  w(y)= \mu(\theta(y))^{-1} w(\sigma(y)).$$

{ Consider the function $g$ on $\Sigma^+\times M$: $$g(y, m)=\mu(m)^{-1} w(y).$$
Then
\begin{multline*} g(\sigma(y),\theta(y) m)=\mu(m)^{-1} \mu(\theta(y)^{-1}) w(\sigma(y))
\\=\mu(m^{-1}) e^{i (-t \cdot r(y) +\langle v, f(y)\rangle -b)} w(y)
=e^{i (-t \cdot r(y) +\langle v, f(y)\rangle -b)}g(y, m).\end{multline*}
Writing $w_0:=g(y,e)$, we have that for all $n$, $g(\sigma^n(y),\theta_n(y))$ lies in the compact set $\{e^{ia}w_0:a\in \br\}$. 

Let $y$ be an element such that the set $\{(\sigma^n(y),\theta_n(y)):n\in \mathbb{N}\}$ is dense in $\Sigma^+\times M$ given by Proposition \ref{den}. 
It follows that $g(\Sigma^+\times M)\subset \{e^{ia} w_0\}$.  This implies that $\mu$ is $1$-dimensional.} \end{proof}

We will use the following simple observation by considering reversing the orientation
of a closed geodesic in $\T^1(X_0)$:
\begin{lem}\label{rev}
For any $y\in  \op{Fix}(\sigma^n)$, there exists $y'\in  \op{Fix}(\sigma^n)$
such that $\tau_n(y)=\tau_n(y')$, $f_n(y)=-f_n(y')$ and $[\theta_n(y)]=[\theta_n(y')]$. \end{lem}

We will repeatedly use the following result of Pollicott \cite[Prop. 2]{Po}:
let $\psi =\mathsf u +i \mathsf v \in C_\beta(\Sigma^+, \c)$ and consider the complex Ruelle operator
 $\mathcal L_{\psi}$ given by
$\mathcal L_\psi (h)(x)=\sum_{\sigma(y)=x} e^{\psi (y)} h(y)$.
Suppose that $\mathcal L_{\mathsf u} 1=1$.
\begin{lem} \label{ppp} For  $0\le a<2\pi$,
 $\mathcal L_{\psi}$  has an eigenvalue $e^{ia +\op{Pr}_\sigma(\mathsf u)}$
 if and only if there exists $\omega\in C(\Sigma^+)$ such that
 $$\mathsf v -a=\omega -\omega\circ\sigma +L$$ where $L: \Sigma^+ \to 2\pi \z$ is a lattice function. 
\end{lem}

\begin{prop}
\label{sp3}
If  $L_{\delta+it, v, \mu}$ has an eigenvalue $e^{ia}$ for some $(v,\mu)\in \bT^d\times \hat M$,
then there exists some integer $p\in \z$ such that $\mu(m)=e^{i p [m]}$ for all $m\in M$, and 
 $\bigcup_{n\in \z} \{t\tau_n(y)-p[\theta_n(y)]+na : y\in \op{Fix}(\sigma^n)\}\subset \pi \mathbb Z$.
\end{prop}

\begin{proof}
Assume that $L_{\delta+it,v,\mu}$ has eigenvalue $e^{ia}$. By Proposition \ref{sp1}, $\mu$ is $1$-dimensional, i.e.
 $\mu(m)=e^{i p [m]}$ for some integer $p\in \z$.
Therefore for $g\in C_\beta(\Sigma^+, \mathbb C)$,
$$L_{\delta+it, v,\mu} (g)(x)=\sum_{\sigma(y)=x} e^{-(\delta+it)\tau (y) + i \langle  v, f(y)\rangle+i p [\theta(y)]} g(y) .$$

By Lemma \ref{ppp}, the function
$$-t\cdot \tau (y) +\langle  v, f(y)\rangle +p [\theta(y)]$$
is cohomologous to a function $a+L(y)$ where $L:\Sigma^+\to 2\pi \z$ is a lattice function. 
Fixing any $y\in \op{Fix}(\sigma^n)$, we have
\begin{equation*}
\label{EQ1}
-t\cdot \tau_n (y) +\langle  v, f_n(y)\rangle +p [\theta_n(y)]-na \in 2\pi \z .
\end{equation*}
 By Lemma \ref{rev},
we have $y'\in \op{Fix}(\sigma^n)$ with $f_n(y')=-f_n(y)$, $\tau_n(y)=\tau_n(y')$ and $[\theta_n(y)] =[\theta_n(y')]$.
%Hence
%$$-t\cdot \tau_n(y) +\langle  v, f_n(y)\rangle +p [\theta_n(y)]-na \in 2\pi \z$$
\begin{equation*}
\label{EQ2}
-t\cdot \tau_n(y) - \langle  v, f_n(y)\rangle +p [\theta_n(y)]-na \in 2\pi \z.
\end{equation*}
Adding the above two terms, we get $-2t \tau_n(y) +2p [\theta_n(y)] -2na \in 2\pi \z$.  This proves the claim.
\end{proof}
%\fi

\begin{thm} \label{sp2} \label{sp4} Let $(t, v, \mu)\in \br \times \bT^d\times \hat M$.
\begin{enumerate}

\item If $L_{\delta,v,\mu}$ has an eigenvalue $e^{ia}$, then $v= 0$ mod $\pi \z^d$. Furthermore, if $\mu=1$ and $v\neq 0$, then $e^{ia}=- 1$; if $\mu\ne 1$,
then $a$ is an irrational multiple of $\pi$. 

\item Let $t\neq 0$. If $L_{\delta+it, v,\mu}$ has an eigenvalue $e^{ia}$, 
then $v=0$ mod $\pi \z^d$ and $a$ is an irrational multiple of $\pi$.
\end{enumerate}
In each case, $e^{ia}$ is a maximal simple eigenvalue.
\end{thm}
\begin{proof}
 Suppose $e^{2\pi a i}$ is an eigenvalue of $L_{\delta + it, v, \mu}$ for some $a \in \mathbb R$. Since $|e^{2\pi a i}| = 1 = e^{\Pr_\sigma(-\delta \tau)}$, the eigenvalue is maximal simple by the complex RPF theorem (see \cite[Thm. 4.5]{PP}).
First note that  $\mu$ is one-dimensional;
 $\mu(m )=e^{i p [m]}$ for some integer $p\in \z$.
By Lemma \ref{ppp},
we have for any $y\in \op{Fix}(\sigma^n)$,
\begin{equation}
\label{lattice eq}
 \langle  v, f_n(y)\rangle +p [\theta_n(y)]-na \in 2\pi \z.
\end{equation}
Using   Lemma \ref{rev}, we get $$ -\langle  v, f_n(y)\rangle +p [\theta_n(y)]-na \in 2\pi \z.$$
By subtracting one from the other, we get $ \langle  v, f_n(y)\rangle\subset \pi \z$.
As $\cup_n\{f_n(y): y\in \op{Fix}(\sigma^n)\}$ generates $\z^d$, it follows that $v=0$ mod $\pi \z^d$.

Now we prove the rest of (1). Suppose $p=0$ and $v\neq 0$. It follows from from (\ref{lattice eq}) that
\begin{equation*}
na\in \pi \mathbb{Z}\,\,\,\text{for all}\,\,\,n\in \mathbb{N}\,\,\,\text{with}\,\,\,\operatorname{Fix}(\sigma^n)\neq \emptyset.
\end{equation*}
Since the transition matrix $\cA$ is aperiodic, $\{n\in \mathbb{N}:\operatorname{Fix}(\sigma^n)\neq \emptyset\}$ contains all sufficiently large integers. Therefore $a=0$ or $\pi$. However $a=0$ implies $v=0$. Hence $a=\pi$.

Suppose $p\ne 0$.
Then 
$$ p [\theta_n(y)]-na \in \pi \z.$$
Since $\{[\theta_n(y)]: y\in \op{Fix}(\sigma^n)\}$ generates $Z(M)$ by \cite[Thm. 1.9]{GR},
$a$ must be an irrational multiple of $\pi$. This shows (1).

In order to show (2), 
if  $a$ were  a rational multiple of $\pi$ and $t \ne 0$,
it follows from Proposition \ref{sp3} that for some integer $p$,
the union $\cup_{n\geq 1}\{t\tau_n(y)-p[\theta_n(y)]: y\in \op{Fix}(\sigma^n)\}$ would be contained in 
$q\pi \z$ for some $q\in \mathbb Q$. If $p=0$ or $Z(M)=\{e\}$, this contradicts Lemma \ref{sha}.
Otherwise, we get 
$\cup_{n\ge 1} \{e^{it\tau_n(y)-ip[\theta_n(y)]}: y\in   \op{Fix}(\sigma^n)\}\subset F$ for some finite subgroup $F$ of $\{e^{i\theta}:\theta\in [0,2\pi)\}$.
This contradicts \cite[Thm. 1.9]{GR}.  
\end{proof}

The following result from the analytic perturbation theory of bounded linear operators is an important ingredient
in our subsequent analysis.
\begin{thm}[Perturbation theorem]\cite{Ka}\label{pt}
Let $B(V)$ be the Banach algebra of bounded linear operators on a complex Banach space
$V$. If $L_0\in B(V)$ has a simple isolated eigenvalue $\lambda_0$ with a corresponding eigenvector
$v_0$, then for any $\e>0$, there exists $\eta>0$ such that if
$L\in B(V)$ with $\|L-L_0\|<\eta$ then $L$ has a simple isolated eigenvalue $\lambda(L)$
and corresponding eigenvector $v(L)$ with $\lambda(L_0)=\lambda_0$, $v(L_0)=v_0$ and such that
\begin{enumerate}
\item $L\mapsto \lambda(L)$, $L\mapsto v(L)$ are analytic for $\|L-L_0\|<\eta$;
\item  for $\|L-L_0\|<\eta$,
$|\lambda(L)-\lambda_0|<\e$ and $\op{spec}(L)-\lambda(L)\subset \{z\in \c : |z-\lambda_0|>\e\} $.
Moreover if $\op{spec}(L_0)-\{\lambda_0\}$ is contained in the interior of a circle $C$ centered at $0$
and $\eta>0$ is sufficiently small, then $\op{spec}(L)-\lambda(L)$ is also contained in the interior of $C$.
\end{enumerate}
\end{thm}

Finally we are ready to prove:
\begin{thm} \label{mrr} Let $\mu\in \hat M$.
Consider the map  $$(s, w)\in \mathbb{R}\times \br^d \mapsto \sum_{n=0}^\infty L_{\delta+is, w,\mu}^n=(1-L_{\delta+is, w,\mu})^{-1}.$$ 
\label{holo}
\begin{enumerate}
\item  Let $\mu\ne 1$. For any $(t,v)\in \br \times \bT^d$, there exists a neighborhood $\mathcal O\subset \mathbb{R}\times \bT^d$ of $(t,v)$ and an analytic map $$(s,w)\in \mathcal O\mapsto H_{s,w,\mu}\in  \op{Hom}(C_\beta(\Sigma^+,W), C_\beta(\Sigma^+,W))$$ such that 
 $(1-L_{\delta+is, w,\mu})^{-1}$ agrees with $H_{s,w,\mu}$ on  $\mathcal{O}\backslash \{(s,w):w=0\,\text{mod}\,\,\pi \z^d\}$.
 \item Let $\mu=1$. For any $(t,v)\in \mathbb{R}\backslash \{0\}\times \bT^d$, there exists a neighborhood $\mathcal O\subset \mathbb{R}\times \bT^d$ of $(t,v)$ and an analytic map $$(s,w)\in \mathcal O\mapsto H_{s,w,\mu}\in  \op{Hom}(C_\beta(\Sigma^+,\c), C_\beta(\Sigma^+,\c))$$ such that 
 $(1-L_{\delta +is, w,\mu})^{-1}$  agrees with $H_{s,w,\mu}$ on  $\mathcal{O}\backslash \{(s,w):w=0\,\text{mod}\,\,\pi\z^d\}$.
\item Let $\mu=1$. There exists a neighborhood $\mathcal O\subset \br \times \br^d$ of $(0,0)$ such that
for all non-zero $(s, w)\in \mathcal O$,
we have  $$(1-L_{\delta+is,  w})^{-1}
=\frac{P_{s,w}} {1-\lambda_{\delta+ is,w}}+Q_{s,w}$$
where $\lambda_{\delta+is, w}$ is the unique eigenvalue of $L_{\delta+ is,  w}$ of maximum modulus
obtained by the perturbation theorem \ref{pt}, $P_{s,w}$ and
$Q_{s,w}$ are analytic maps from $\mathcal O$ to $ \op{Hom}(C_\beta(\Sigma^+,\c), C_\beta(\Sigma^+,\c))$.

\end{enumerate}
\end{thm}

\begin{proof}
 If $\sigma_0(L_{\delta+it, v,\mu})<1$, then, by the perturbation theorem \ref{pt},
  there is a neighborhood $\mathcal O$ of $(t,v)$ in $\mathbb{\br}\times \mathbb{\br}^d$ such that 
 $\sigma_0(L_{\delta+is,w,\mu})<1$ for any $(s,w)\in \mathcal O$.  This implies that $\sum_{n=0}^\infty L_{\delta+is, w,\mu}^n$ converges absolutely and hence analytic on $\mathcal{O}$.
 
 Now suppose $\sigma_0(L_{\delta+it, v,\mu})=1$.
In any of the following three case (1) $\mu \ne 1$,  (2) $\mu=1$ and $t\ne 0$ or (3)  $\mu=1, t=0, v\in \bT^d\backslash \{0\}$,
by Theorem \ref{sp2}, $L_{\delta+it, v,\mu}$ has a simple eigenvalue $e^{ia}$ of maximum modulus
 with $a$ some 
 irrational multiple $\pi$ or $\pi$.
By Theorem \ref{pt}, 
there exists a neighborhood $\mathcal O\subset \br\times \br^d$ of $(t,v)$ such that
for any $(s,w)\in \mathcal O$,
 $L_{\delta+is, w,\mu}$ can be written as 
$$L_{\delta+is, w,\mu}=  \lambda_{\delta+is,w,\mu} P_{\delta+is,w,\mu}+N_{\delta+is,w,\mu}$$
 where $\lambda_{\delta+is, w\mu}$ is the simple maximal eigenvalue of $L_{\delta+is,w,\mu}$, $P_{\delta+is,w}$
is the eigenprojection to the eigenspace associated to $\lambda_{\delta+is,w,\mu}$ and
$\sigma_0(N_{\delta+is,w,\mu})<1$. Moreover, $\lambda_{\delta+is,w\mu}$, $P_{\delta+is,w,\mu}$, $N_{\delta+is,w,\mu}$ are analytic on $\mathcal{O}$. Hence for all $n\in \N$,
$$L^n_{\delta+is, w,\mu}=\lambda_{\delta+is, w,\mu}^n P_{\delta+is,w,\mu}+N^n_{\delta+is,w,\mu}.$$
By choosing $\mathcal{O}$ sufficiently small, we have $\sum_n N^n_{\delta+is,w,\mu}$ converges absolutely on $\mathcal{O}$, and the map 
 \be \label{lll} 
\frac{P_{\delta+is,w,\mu}}{1-\lambda_{\delta+is, w,\mu }} + \sum_n N^n_{\delta+is,w,\mu} \ee 
is analytic on $\mathcal O$. Note that on $\mathcal{O}\backslash\{(s,w):w=0\,\text{mod}\,\,\pi \z^d\}$, by  Proposition \ref{sp1} and Theorem \ref{sp2}, the spectral radius of $L_{\delta+is,w,\mu}$ is strictly less than $1$. Hence $\sum_n L^n_{\delta+is, w,\mu} $ converges absolutely on this set. We have  \eqref{lll} agrees with $(1-L_{\delta+is, w,\mu})^{-1}$ on $\mathcal{O}\backslash\{(s,w):w=0\,\text{mod}\,\,\pi \z^d\}$.

Suppose $(t, v, \mu)=(0,0,1)$. Then the map
$\sum\lambda_{\delta+is, w,\mu}^n = ({1-\lambda_{\delta+ is,w,\mu}})^{-1}$ 
is analytic on $\mathcal O-\{(0,0)\}$ and hence
 \eqref{lll} is
  analytic on $\mathcal O-\{(0,0)\}$.
 This finishes the proof.
\end{proof}

\subsection{ Asymptotic expansion.} For each $u\in \br^d$ close to $0$,
there exists a unique $$P(u)\in \br$$ such that
the pressure of the function
 $x \mapsto -P (u) \tau(x) + \langle u, f(x)\rangle$ on $\Sigma$ is $0$.
 Moreover $P(0)=\delta$, $\nabla P(0)=0$, the map $u\to P(u)$
 is analytic, and the matrix $\nabla^2 P(0)=(\frac{\partial^2P}{\partial u_i\partial u_j}(0))_{d\times d}$ is a
  positive definite matrix  (cf. \cite[Lem. 8]{PS}).
  
Set \begin{equation}\label{ds} \sigma=\op{det}(\nabla^2 P(0))^{1/d} .\end{equation}

Let $$C(0)= \int_{\br^d} e^{-\frac 12 w^t \nabla^2 P(0) w }dw=\left(\frac{2\pi}{\sigma}\right)^{d/2}.$$

\begin{remark} 
It follows from (\cite{KS2} and \cite{Rat}) that, for $X_0$ compact,
 the distribution $\frac{\xi(x: t)}{\sqrt t}$ as $x$ ranges over the image of $\mathcal F$ in $\T^1(X_0)$
converges to the distribution of a multivariable Gaussian random variable $N$ on $\br^d$
with a positive definite covariance matrix 
 $ \op{Cov}(N)= \nabla^2 P(0)  $.
\end{remark}

{\begin{Def}\label{defL}
Let $\mathcal{L}$ be the family of functions on $\Sigma^+\times \br$
which are of the form $\Phi\otimes u$ where 
$\Phi\in C_\beta(\Sigma^+)$ and $u\in C_c(\mathbb{R})$.
\end{Def}

In the rest of this subsection, we fix $(\mu, W)\in \hat M$ and $w\in W$.

For each $T>1$ and $\Phi\otimes u\in \mathcal{L}$, define the $W$-valued functions $$\mathsf Q^{(n)}_{\mu, w, T}(\Phi\otimes u)\quad \text{ and }\quad  \mathsf Q_{\mu, w,T}(\Phi\otimes u)$$ 
on $\Sigma^+\times \z^d\times M$
as follows: for $(x,\xi, m)\in \Sigma^+\times \z^d\times M$,
\begin{multline*} 
\Q^{(n)}_{\mu,w,T}(\Phi\otimes u)(x, \xi,m)
\\=\frac{1}{2\pi}\int_{t\in \br}  e^{-i T t} \hat u(t) \left(\int_{v\in \bT^d} e^{i \langle v, \xi\rangle } L_{\delta-it, v ,\mu}^n  (\Phi h\mu(m)w)   (x)  dv \right) dt  \end{multline*}
and
\begin{multline*} \Q_{\mu, w,T}(\Phi\otimes u)(x, \xi,m)
\\=\frac{1}{2\pi} \int_{(t,v)\in \br\times \bT^d}  e^{-i T t+\langle v, \xi \rangle} \hat u(t)  
\cdot \left( \sum_{n\ge 0}  L_{\delta-it, v ,\mu}^n  (\Phi h\mu(m)w)   (x) \right)  dv dt  .\end{multline*}
Here $\hat u(t)=\int_{\br} e^{-ist} u(s)ds$.
We set $\Q_{T}=\Q_{1,1,T}$ and $\Q_T^{(n)}=\Q_{1,1,T}^{(n)}$.

\begin{thm}\label{pq}  Let $\Phi\otimes u\in \mathcal{L}$ and $(x,\xi, m)\in \Sigma^+\times \z^d\times M$.
 \begin{enumerate}
\item For each $T>0$,
$$\sum_n \Q_{\mu,w,T}^{(n)}(\Phi\otimes u)(x, \xi, m)=\Q_{\mu,w,T} (\Phi\otimes u)(x,\xi, m) $$
where the convergence is uniform on compact subsets.

\item  We have $$\lim_{T\to +\infty} {T^{d/2}} \Q_{T}(\Phi\otimes u)(x,\xi, m) =
 \frac{ \hat u(0)  C(0)  }{\int \tau d \nu}   \rho(\Phi  h ) h(x)$$
where the convergence is uniform on compact subsets.
\item For any non-trivial $(\mu, W)\in \hat M$ and $w,w'\in W$, 
we have
$$\lim_{T\to +\infty}  \langle {T^{d/2}} \Q_{\mu, w,T} (\Phi\otimes u)(x,\xi, m), w'\rangle = 0 $$
where the convergence is uniform on compact subsets.
\end{enumerate}\end{thm}
\begin{proof}
In proving this theorem,
we may assume that
 the Fourier transform $\hat{u}$ belongs to $C_c^N(\mathbb{R})$ for some $N\geq \frac{d}{2}+2$ (see \cite[Lemma 2.4]{BL2}). 
For (1), it is sufficient to show that 
$\|\hat{u}(t)\sum_{n\geq N}L^n_{\delta-it,v,\mu}(\Phi h\mu(m)w)(x)\|$ is
 dominated by a single absolutely integrable function of $(t,v)$ almost everywhere.

We have
\begin{align}
\label{convergence}
&\|\hat{u}(t)\sum_{n\geq N}L^n_{\delta-it,v,\mu}(\Phi h\mu(m))(x)\|\\
\leq & |\hat{u}(t)|
\cdot \| L^N_{\delta-it,v,\mu}\| \cdot \|
 \sum_{n=0}^{\infty}L^n_{\delta-it,v,\mu}(\Phi h\mu(m)w)(x) \| .\nonumber
\end{align}

When $\mu$ is nontrivial, $\sum_{n=0}^{\infty} L^n_{\delta-it,v,\mu}$ agrees almost everywhere with the operator $H_{\delta-it,v,\mu}$ which is described in Theorem \ref{holo}. Noting that  $\lVert H_{\delta-it,v,\mu}\rVert$ is bounded on compact sets (e.g. $\operatorname{supp}(\hat{u})\times \T^d$), we have
\begin{align*}
(\ref{convergence})\ll  \lVert (\Phi h\mu(m)w ) x\rVert \cdot |\hat{u}(t)|,
\end{align*}
verifying (1) for the case when $\mu$ is nontrivial. We refer to  Step 7 of \cite[Appendix]{LS} for the proof of (1) for $\mu$ trivial. 

To prove (2), consider the function
$$F(t, v)=   e^{i \langle v, \xi\rangle }  \sum_{n\ge 0}  L_{\delta-it, v }^n  (\Phi h)   (x)   $$
so that $$ \Q_{T}(\Phi\otimes u)(x, \xi,m)
\\=\frac{1}{2\pi} \int_{(v,t)\in \bT^d\times \br} e^{-i T t}  \hat u(t) F(t,v) dv dt  .$$

Let $\mathcal O\subset \br \times \br^d$ be a neighborhood of $(0,0)$ as in Theorem \ref{mrr} (3)
and choose any $C^{\infty}$-function  $\kappa(t, v)=\kappa_1(t)\kappa_2(v)$ supported in $\mathcal O$. 

Since $F(t,v)$ is analytic outside $\mathcal O$ and $\hat u\in C^N(\br)$,
the following value of the Fourier transform is at most $O(T^{-N})$:
\begin{equation}\label{ft1}
 \int_{t\in \br} e^{-iTt} \left( \int_{v\in \bT^d} \hat u(t) (1-\kappa(v,t)) F(t,v)dv\right) dt=O(T^{-N}).
\end{equation}

We now need to estimate
$$ \int_{t\in \br} e^{-iTt} \left( \int_{v\in \bT^d} \hat u(t) \kappa(v,t)F(t,v)dv\right) dt .$$
This can be done almost identically to Step 5 in the appendix of \cite{LS};
 we give a brief sketch of their arguments here for readers' convenience.
On $\mathcal O$,
we can write  $$F(t,v)
=e^{i\langle v, \xi\rangle}  \frac{P_{t,v}(\Phi h) (x) } {1-\lambda_{\delta-it,v}}+Q_{t,v}(\Phi h )(x)$$
where $\lambda_{\delta-it, v}$,
$P_{t,v}$ and $Q_{t,v}$ are as described in Theorem \ref{mrr} (3).

 Applying Weierstrass preparation theorem to $1-\lambda_{\delta-it,v}$, we have
 that for $(t, v)\in \mathcal O$,
\begin{equation}\label{We}
1-\lambda_{\delta-it,v}=A(t,v)(\delta-it-P(v)),
\end{equation}
where  $A$ is non-vanishing and analytic in $\mathcal O$, by replacing $\mathcal O$ by a smaller neighborhood
if necessary.

We have $P(0)=\delta$,  $P(v)= P(0) -\tfrac 12 v^t \nabla^2 P(0) v +o(\|v\|^2)$ for $v$ small,
and $ A(0, 0)= 
-\frac{d\lambda (s,0)}{ds}\lvert_{s=\delta}=  \int \tau d\nu$.
Set $R(v)=\delta-P(v)$.

For $(t,v)\in \mathcal O$,
set
\begin{equation}
\label{eq 5}
a(t,v):=\frac{-\kappa_1(t)\kappa_2(v)\hat{u}(t)P_{t,v}(\Phi h)(x)}{A(t,v)}.
\end{equation}

Suppose for now that $a(t, v)$ is of the form $ c_x(t) b(v)$.
Using $1/z=-\int_0^{\infty} e^{T' z} dT'$ for $\Re(z)<0$, we get 
\begin{align*}
\label{eq 4}&  \int_{t\in \br} e^{-iTt} \left( \int_{v\in \bT^d} \hat u(t)\kappa(t,v) F(t,v)dv\right) dt\\
=& \int_{t\in \br} e^{-iTt}  \left( \int_{v\in \bT^d}\frac{a(t,v)e^{i\langle \xi,v\rangle}}{it  - R(v)}dv\right) dt +O(T^{-N})\nonumber
\\ =& -\int_{t\in \br} e^{-iTt}  \left( \int_{v\in \bT^d}
{a(t,v)e^{i\langle \xi,v\rangle}}\int_0^{\infty} e^{(it -R(v))T' } dT'dv\right) dt +O(T^{-N}) \nonumber
\\ =& -\int_0^{\infty}  \int_{t\in \br}  e^{-i(T-T')t} \left( \int_{v\in \bT^d} a(t,v)  e^{i\langle \xi,v\rangle -R(v)T' } dv \right)dt dT' +O(T^{-N}) \nonumber
\\ = &- \int_{T/2}^\infty \hat c_x (T-T')  
\int_{v\in \bT^d}  b(v)  e^{i\langle \xi,v\rangle - R(v)T' } dv dT' +O(T^{-N}) \nonumber
 =  - \int_{-\infty}^{T/2}  \tfrac{\hat c_x (T'') }{(T-T'')^{2/d}} 
\\ &   \int_{v\in\sqrt {T-T''} \bT^d} b(\tfrac{v}{\sqrt {T-T''} })  
e^{i\langle \xi,\tfrac{v}{\sqrt {T-T''} }\rangle-R(\tfrac{v}{\sqrt {T-T''} } )(T-T'') } dv dT''+O(T^{-N}) \nonumber
.\end{align*}

Using  $R(v)=\tfrac{1}{2} v^t \nabla^2 P(0) v +o(\|v\|^2)$, and $C(0)=\int_{\br^d} e^{-\tfrac{1}{2} v^t \nabla^2 P(0) v }
dv$,
the above is asymptotic to
$$ \int_{-\infty}^{T/2}  (T-T'')^{-2/d} \hat c_x (T'') C(0) b(0)dT''
= T^{-d/2}  (2\pi c_x(0) C(0) b(0) +o(1)).$$

By approximating $a(t,v)$ by a  sum of functions of  the form
$c_x(0) b(v)$ using Taylor series expansion,
one obtains the following estimation:

\be\label{ftt}
\lim_{T\to \infty}T^{d/2} \int_{t\in \br} e^{-iTt} \left( \int_{v\in \bT^d} \hat u(t)\kappa(t,v) F(t,v)dv\right)=
 \frac{2\pi \hat{u}(0)C(0)}{\int \tau d\nu} \rho(\Phi h)h(x).
\ee 

Therefore, putting \eqref{ft1} and \eqref{ftt} together, we deduce
$$
\lim_{T\to \infty} T^{d/2}\Q_{T}(\Phi\otimes u)(x,\xi,m)
= \frac{\hat{u}(0)C(0)}{\int \tau d\nu} \rho(\Phi h)h(x), $$ verifying (2).

 For (3), we have
 \begin{align*}
 &\langle \Q_{\mu,w,T}(x,\xi,m)(\Phi\otimes u),w'\rangle\\
 =&\frac{1}{2\pi}\int_{t\in \mathbb{R}} e^{-iTt}\hat{u}(t)\int_{v\in \bT^d} e^{i\langle v,\xi \rangle}\langle (I-L_{\delta-it,v,\mu})^{-1}(\Phi h\mu(m)w)(x),w' \rangle dvdt.
 \end{align*}
 Hence by Theorem \ref{mrr}, and the assumption that
 $\hat{u}$ is of class $C^N$ ($N\geq d/2+2$), the Fourier transform decays as:
\begin{equation*}
\langle  \Q_{\mu,w,T}(x,\xi,m)(\Phi\otimes u),w'\rangle=O(T^{-N})
\end{equation*}
which implies (3).
\end{proof}}

\section{Local mixing and matrix coefficients for local functions}
We retain the assumptions and notations from section 3.
Recall the BMS measure $\mathsf{m}^{\BMS}$ on $\Gamma\ba G$,
and its support $\Omega\subset \Gamma\ba G$.
In this section, we study the asymptotic behavior of the correlation functions
$$\la a_t \psi_1, \psi_2 \ra_{\mathsf{m}^{\BMS}} :=\int_{\Omega} \psi_1(xa_t) \psi_2(x) d\mathsf{m}^{\BMS}(x)
$$
and $$\la a_t \psi_1, \psi_2 \ra :=
 \int_{\Gamma\ba G} \psi_1(xa_t) \psi_2(x) d\mathsf{m}^{\Haar}(x)$$
for $\psi_1, \psi_2\in C_c(\Gamma\ba G)$.

\subsection{ Correlation functions for  $(\Omega, a_t, \mathsf{m}^{\BMS})$.}
 We use the suspension flow model for $(\Omega, a_t, \mathsf{m}^{\BMS})$ which
was constructed in Section 2. That is, we identify
the right translation action of $a_t$ on
$\Omega$ with the suspension flow
on  $$\Sigma^{f,\theta,\tau}:=\Sigma\times \Z^d \times M\times  \br /\sim
$$
where $\sim$ is given by
$\zeta (x,\xi, m, s)=(\sigma x, ,  \xi+f(x) ,\theta^{-1} (x) m, s-\tau(x) ) $.

We write  
$$\tilde \Omega:=\Sigma \times\Z^d \times M\times  \br, \quad
\tilde \Omega^+:=\Sigma^+\times\Z^d \times M\times  \br$$
and 
$$\Omega^+:=\Sigma^+\times \Z^d \times M\times  \br/\sim .$$

Consider the product  measure  on $\tilde \Omega$:
$$d{ \tilde{\mathsf M}}:=\frac{1}{\int \tau d\nu} (d\nu  d {\xi}  dm ds).$$

Recall that the BMS measure $\mathsf{m}^{\BMS}$  on $\Omega$ corresponds to the  measure  $\mathsf M$ on $\Sigma^{f,\theta,\tau}$ induced
by $\tilde{\mathsf M}$.

\begin{Def}
Let $\mathcal{F}_0$ be the family of functions on $\tilde \Omega^+$ which are of the form
$$\Psi (x, \xi, m, s)=\Phi (x)  \delta_{\xi_0}(\xi)  u (s)  \langle  \mu(m)w_1, w_2\rangle$$ where 
$\Phi\in C_\beta(\Sigma^+)$, $u\in C_c(\mathbb{R})$,
 $\xi_0\in \Z^d$, $(\mu, W)\in \hat M$ and $w_1, w_2\in W$ are unit vectors.
We will write $\Psi=\Phi\otimes \delta_{\xi_0}  \otimes
u\otimes  \langle  \mu(\cdot )w_1, w_2\rangle $.
\end{Def}

For $\Psi_1, \Psi_2\in C_c(\tilde \Omega^+)$, define
$$I_t(\Psi_1, \Psi_2):=
\sum_{n=0}^{\infty}\int_{\tilde \Omega}\Psi_1\circ \zeta_n(x,\xi,m, s+t)\cdot \Psi_2(x,\xi,m,
s)\; d{ \tilde{\mathsf M}}(x,\xi, m,s)$$
where $ \zeta_n(x,\xi,m, s)=(\sigma^n(x), \xi+f_n(x),\theta_n^{-1}(x)m, s-\tau_n(x))$.

\begin{lem}\label{it}  Let 
$\Psi_2=\Phi\otimes \delta_{\xi_0}  \otimes
u\otimes  \langle  \mu(\cdot )w_1, w_2\rangle \in \mathcal F_0 .$
Then for any  $\Psi_1\in C_c(\tilde \Omega^+)$,
\begin{multline*} I_t(\Psi_1, \Psi_2)= \\
 \tfrac{1}{(2\pi)^d \int \tau d\nu} \int \Psi_1 (x, \xi_0-\xi, m, s) \cdot \langle \Q_{\mu, w_1, t-s}(\Phi\otimes u)(x, \xi, m), w_2 \rangle 
  d\xi d\rho(x) ds dm.\end{multline*}
\end{lem}
\begin{proof} 
Since $d\nu(x)=h(x)d\rho(x)$, we have
 \begin{align*}
&\int \tau d\nu \cdot I_t(\Psi_1,\Psi_2) =\\
& \sum_{n=0}^{\infty} \int\Psi_1(\sigma^n x, \xi+f_n(x),\theta_{n}^{-1}(x) m,s-\tau_n(x))\,\Psi_2(x,\xi,m,s-t)
 h(x)d\rho(x) ds dm d\xi . \\
\end{align*}
Since $d\rho$ is an eigenmeasure of $L_{\delta}$ with eigenvalue $1$,
$$\int_{\Sigma^+} (L_{\delta}^n F)(x) d\rho(x)= \int_{\Sigma^+} F(x) d\rho(x).$$
Using this, the above
is equal to
 \begin{multline*} \sum_{n=0}^{\infty} \int_{\tilde{\Omega}^+}\Psi_1(x,\xi,m,s)\sum_{\sigma^n y=x} e^{-\delta \tau_n(y)} (\Phi\cdot h) (y) \delta_{\xi_0}(\xi-f_n(y)) \\
u(s-t+\tau_n(y))\langle \mu(\theta_n(y)m) w_1,w_2\rangle d\rho(x) d\xi ds d m.\end{multline*}
Using the identity
$\delta_\xi(f_n(y))= \frac{1}{(2\pi)^d} \int_{\bT^d} e^{i \la w, \xi-f_n(y)\ra} dw$
and  the  Fourier inversion formula of $u$: $u(t)=\frac{1}{2\pi}\int_{\br} e^{ist}\hat u(s) ds$,
the above is  again equal to
$$\frac{1}{(2\pi)^d } \ \int \Psi_1 (x,\xi_0-\xi, m, s)
 \cdot \langle \Q_{\mu, w_1, t-s}(\Phi\otimes u)(x, \xi, m), w_2 \rangle 
 d\xi d\rho(x) ds dm .$$
This proves the claim.
 \end{proof}

\bigskip

\begin{prop}\label{m2} For $\Psi_1, \Psi_2\in C_c(\tilde \Omega^+)$, we have 
\be \label{prove}\lim_{t\to +\infty}t^{d/2} I_t(\Psi_1, \Psi_2) = \frac{1}{(2\pi\sigma)^{d/2}   } { \tilde{\mathsf M}}(\Psi_1)  { \tilde{\mathsf M}} (\Psi_2).\ee
\end{prop}

\begin{proof}  Let $\Psi_1 \in C_c(\tilde \Omega^+)$.

\noindent{\bf Step 1:} Let $\mathcal F$ be the space of functions which are finite linear combinations of functions from $\mathcal F_0$.
We first show \eqref{prove} holds for any  $\Psi_2\in \mathcal F$.
It suffices to
consider the case
where $$\Psi_2(x, \xi, m,s)=\Phi (x)\otimes \delta_{\xi_0}(\xi) \otimes u (s)\cdot  \langle  \mu(m)w_1, w_2\rangle \in \mathcal F_0.$$

Let $$F_t(x,\xi, m, s):=\tfrac{t^{d/2}}{ \int \tau d\nu \cdot  (2\pi)^{d}} 
 \Psi_1(x, \xi_0-\xi,m, s)\la  \Q_{\mu,w_1, t-s}(\Phi\otimes u) (x,\xi, m), w_2\ra.$$ 
Then Lemma \ref{it} gives
$$t^{d/2} I_t(\Psi_1, \Psi_2)= \int F_t(x, \xi, \theta, s)  d\rho(x) dsd\xi dm  .$$

We consider two cases. First suppose $\mu=1$. Then 
Theorem \ref{pq}(2) implies that
$F_t (x,\xi, m,s) $ is dominated by a constant multiple of $\Psi_1$ and
 converges pointwise
to an $L^1$-integrable function on $\tilde \Omega$:
$$ \tfrac{1}{( \int \tau dv )^2(2\pi)^{d}} C(0) \hat u(0) \rho (\Phi h) h(x) \Psi_1(x, \xi_0-\xi, m,s) . $$
Hence by the dominated convergence theorem,
\begin{align*} &\lim_{t\to \infty} t^{d/2} I_t(\Psi_1, \Psi_2)
\\ &= \frac{1}{ (2\pi)^{d} ( \int \tau  dv )^2} C(0) \hat u(0) \rho (\Phi  h) \int \Psi_1(x, \xi_0-\xi,m,  s)  
 h(x) d\rho(x) ds d\xi dm \\ &
 = \frac{ C(0) }{(2\pi)^{d}} \int \Psi_1 d{ \tilde{\mathsf M}} \cdot   \int \Psi_2 d{ \tilde{\mathsf M}}.
\end{align*}

 Plugging $C(0)={(2\pi/\sigma)^{d/2}}$ in the above,
 we get $$t^{d/2} I_t(\Psi_1, \Psi_2)\sim \frac{1}{(2\pi\sigma)^{d/2}} \int \Psi_1 d{ \tilde{\mathsf M}} \cdot   \int \Psi_2 d{ \tilde{\mathsf M}}.$$
  
Now suppose $\mu$ is non-trivial.  Then $d{ \tilde{\mathsf M}}(\Psi_2)=0$. On the other hand, 
Theorem \ref{pq}(3) implies that  $F_T$ converges to $ 0$ pointwise, and  is dominated by $\Psi_1$.
Therefore, by the dominated convergence theorem, we get 
$$\lim_{T\to \infty}t^{d/2} I_t(\Psi_1, \Psi_2)=0$$
proving the claim.

As a consequence, we have
\begin{equation}\label {fi}  \limsup_t t^{d/2} |I_t(\Psi_1, \Psi_2)| <\infty \end{equation}
for any $\Psi_1\in C_c(\tilde \Omega)$ and $\Psi_2\in \mathcal F$.

\noindent{\bf Step 2:} 
 Let $\Psi_2\in C_c(\tilde \Omega^+)$ be a general function. 
 For any $\e>0$,  there exist  $F_2,  \omega_2 \in \mathcal F$ such that for any $(x,\xi, m,s)\in \tilde \Omega^+$,
 \be \label{fw}  |\Psi_2(x,\xi,m, s)-F_2(x,\xi, m,s)|  \le \e \cdot \omega_2(x,\xi, m,s) .\ee
 First to find $F_2$, 
 the Peter-Weyl theorem implies that
 $\Psi_2(x,\xi, m, s)$ can be approximated
 by a linear combination of functions of form 
 $$\kappa(x, \xi, s)\la \mu (m) w_1, w_2\ra$$
 for $\kappa\in C_c(\Sigma^+\times \z^d\times \br)$ and $(\mu, W)\in \hat M$ and $w_1, w_2\in W$ unit vectors.
 As $\z^d$ is discrete, $\kappa$ can be approximated by linear combinations
 of functions of form $c(x, s)\delta_{\xi_0}$ with $c(x,s)\in C_c(\Sigma^+\times \br)$. 
 Now $c(x,s)$ can be approximated by linear combinations
 of functions of form $\Phi(x) u(s)$ with $\Phi\in C_\beta(\Sigma^+)$ and $u\in C_c^\infty(\br)$ by
 the Stone-Weierstrauss theorem. This gives that for any $\e>0$,
 we can find  $F_2\in \mathcal F$ such that
 $$ \sup |\Psi_2(x,\xi,m, s)-F_2(x,\xi, m,s)|  \le \e.$$

 Now let $\mathcal O$ be  the union of the supports
  of $F_2$ and $\Psi_2$, $\mathcal O'$ be the $1$-neighborhood of $\mathcal O$, and let $\kappa:=\| F_2\|_\infty
  +\|\Psi_2\|_\infty +1$.
  We can then find $\omega_2\in \mathcal F$ such that $\omega_2=\kappa$ on $\Omega$
  and $\omega_2=0$ outside $\Omega'$.
  Then for any $(x,\xi, m,s)\in \tilde \Omega^+$,
  $$  |\Psi_2(x,\xi,m, s)-F_2(x,\xi, m,s)|  \le \e \omega_2(x,\xi, m, s)$$
as required in \eqref{fw}.

\noindent{\bf Step 3:}  By Step (1) and (2),
we have  $$\limsup |t^{d/2} I_t(\Psi_1, \Psi_2) -t^{d/2} I_t(\Psi_1, F_2)|  \le\e  \limsup
 t^{d/2} | I_t(\Psi_1, \omega_2)| \le  \e c_0,$$
where $c_0:=\limsup_t t^{d/2} |I_t(\Psi_1, \omega_2)| <\infty$.
 
 Hence  
 \begin{align*}
  &\lim_t t^{d/2} I_t(\Psi_1, \Psi_2)\\
 = &\lim_t t^{d/2} I_t(\Psi_1, F_2) + O(\e)\\
  = &\frac{1}{(2\pi \sigma)^{d/2}} { \tilde{\mathsf M}} (\Psi_1) { \tilde{\mathsf M}}(F_2) +O(\e)\\
  =  & \frac{1}{(2\pi \sigma)^{d/2}} { \tilde{\mathsf M}}(\Psi_1){ \tilde{\mathsf M}} (\Psi_2) +O(\e),
 \end{align*}
 since ${ \tilde{\mathsf M}}(F_2)={ \tilde{\mathsf M}}(\Psi_2)+O(\e)$.
 
 As $\e>0$ is arbitrary, this proves the claim for any $ \Psi_2\in C_c(\tilde \Omega^+)$.
 \end{proof}

  \begin{thm} \label{m1} Let $\psi_1, \psi_2\in C_c(\Omega)$. Then
 $$\lim_{t\to \infty} {t^{d/2}} \int \psi_1(ga_t) \psi_2 (g) d\mathsf{m}^{\BMS}(g)
 = \frac{1}{(2\pi\sigma)^{d/2}}  \cdot \mathsf{m}^{\BMS}(\psi_1) \mathsf{m}^{\BMS}(\psi_2).$$
\end{thm}
\begin{proof}
For each $i=1,2$, let $\Psi_i\in C_c(\tilde \Omega)$ be the lift of $\psi_i$ to $\tilde \Omega$ so
 that  $$\psi_i [(x,\xi,m, s)]=\sum_{n\in \mathbb{Z}}\Psi_i\circ \zeta^n(x,\xi,m, s). $$

We assume that the support of $\psi_2$ (and hence of $\Psi_2$) is small enough so that
$\Psi_2(\zeta^n(x,\xi, m,s))=0$ for all $n \ne 0$ if $(x,\xi, m,s)\in \text{supp}(\Psi_2)$.

We first claim that  for all large $t\gg 1$,
\be\label{unf} \int_{\Omega} \psi_1 (ga_t) \psi_2(g) d\mathsf{m}^{\BMS}(g) = I_t(\Psi_1, \Psi_2) .\ee

 Using the unfolding,
\begin{align*} &\int_{\Omega} \psi_1 (ga_t) \psi_2(g) d\mathsf{m}^{\BMS}(g)\\
=&
\sum_{n=-\infty}^{\infty}\int_{\text{supp}(\Psi_2)}
\Psi_1\circ \zeta^n(x,\xi,m, s+t)\cdot \Psi_2(x,\xi,m,  s) \, d{ \tilde{\mathsf M}}
\\=&
\sum_{n=0}^{\infty}\int_{\tilde \Omega}\Psi_1\circ \zeta^n(x,\xi,m, s+t)\cdot \Psi_2(x,\xi,m,  s) \, d{ \tilde{\mathsf M}} \\
+&\sum_{n=1}^{\infty}\int_{ \tilde \Omega}\Psi_1\circ \zeta^{-n}(x,\xi,m ,s+t)\cdot \Psi_2(x,\xi,m, s) \, d{ \tilde{\mathsf M}} . \end{align*}
The first term of the last equation is $I_t(\Psi_1, \Psi_2)$. For the second term, note that
 for any $(x,\xi,m, s)\in \text{supp}(\Psi_2)$,
$$\Psi_1\circ \zeta^{-n}(x,\xi,m ,s+t)=\Psi_1(\sigma^{-n}(x), \xi-f_n(x), \theta_n^{-1}(x)m, s+t+\tau_n(x))$$
which is $0$ if $t$ is large enough, as $\tau_n(x)>0$.

Therefore the second term is $0$ for $t$ large enough, proving the claim \eqref{unf}.
Therefore  if $\Psi_1, \Psi_2\in C_c(\tilde \Omega^+)$,
Theorem \ref{m1} follows from Proposition \ref{m2}.

 Let $\Psi_1, \Psi_2 \in C_c(\tilde \Omega)$. Then for any $\e >0$,
 we can find a sufficiently large $k\ge 1$, $F_1,F_2,\omega_1,$ and $ \omega_2$ in $C_c(\tilde\Omega^+)$ such that
for all $(x, \xi,m, s)\in \tilde \Omega$, \begin{align*}
&|\Psi_i\circ \zeta^k(x,\xi,m,s)-F_i(x,\xi+f_k(x),\theta_k^{-1}(x)m,s-\tau_k(x))|\\
<&\epsilon\cdot  \omega_i(x,\xi-f_k(x),\theta_k^{-1}(x)m,s-\tau_k(x)).
\end{align*}
We then deduce by applying the previous case to $F_i$ and $\omega_i$ that
\begin{align*}
&\lim t^{d/2}\int \psi_1(ga_t)\psi_2(g)d\mathsf{m}^{\BMS}(g)\\
=&\lim t^{d/2}I_t(\Psi_1,\Psi_2)\\
=&\lim t^{d/2}I_t(\Psi_1\circ \zeta^k,\Psi_2\circ \zeta^k)\\
=&\lim \left( t^{d/2}I_t(F_1,F_2)+O(\epsilon\cdot t^{d/2} (I_t(F_1,\omega_2)+I_t(F_2,\omega_1)+I_t(\omega_1,\omega_2))) \right) \\
=&\tfrac{1}{(2\pi\sigma)^{d/2}}  { \tilde{\mathsf M}} (\Psi_1)  { \tilde{\mathsf M}} (\Psi_2)+O(\e) .
\end{align*}
As   ${ \tilde{\mathsf M}} (\Psi_i) =  \mathsf{m}^{\BMS} (\psi_i)$ and $\e>0$ is arbitrary, this finishes the proof.
 \end{proof}

 \begin{Rmk}
 \rm We remark that the methods of our proof of Theorem \ref{m1} can be extended to other Gibbs measures on $\Gamma\ba G$.
 \end{Rmk}
 \subsection{ Correlation functions for  $(\Gamma\ba G, a_t, \mathsf{m}^{\Haar})$.}

We can deduce
the asymptotic of the correlation functions for the Haar measure Theorem \ref{mix2}
from that for the BMS measure Theorem \ref{m1} via the following theorem:

\begin{thm}\label{com}
Suppose that there exists a function $J:(0,\infty)\to (0,\infty)$ such that for any $\psi_1, \psi_2\in C_c(\G\ba G)$,
\be \label{e1} \lim_{t\to + \infty} J(t) \int \psi_1(ga_t) \psi_2(g) d\mathsf{m}^{\BMS} (g)=\mathsf{m}^{\BMS}(\psi_1)\mathsf{m}^{\BMS}(\psi_2).\ee
Then  
for any $\psi_1, \psi_2\in C_c(\G\ba G)$,
\be \label{e2} \lim_{t\to +\infty} J(t)e^{(D-\delta)t}  \int  \psi_1(ga_t) \psi_2(g) d\mathsf{m}^{\Haar}(g) =\mathsf{m}^{\BR_+}(\psi_1)\mathsf{m}^{\BR_-}(\psi_2).\ee
\end{thm}

The main idea of this theorem appeared first in Roblin's thesis \cite{Ro} and was further developed and used in 
(\cite{OS}, \cite{MO}, \cite{OW}).
The key ingredients of the arguments are the product structures of the measures $m^{\BMS}$ and $m^{\Haar}$ and
the study of the transversal intersections for the translates of horospherical pieces by the flow $a_t$.
The verbatim repetition of
 the proof of \cite[Theorem 5.8]{OW} while replacing $H$ by $N^-AM$
proves Theorem \ref{com}.

 Using theorem \ref{com},  we deduce the following  from Theorem \ref{m1}:
\begin{thm} \label{mix22}
 Let $\psi_1, \psi_2\in C_c(\Gamma \ba G)$.
Then
$$ \lim_{t\to +\infty} {t^{d/2} e^{(D-\delta)t}} \int_{\Gamma \ba G}\psi_1(ga_t) \psi_2(g) \; d\mathsf{m}^{\Haar}(g)= 
 \frac{\mathsf{m}^{\op{BR_+}}( \psi_1) \mathsf{m}^{\op{BR}_-}( \psi_2)}{(2\pi\sigma)^{d/2} \mathsf m^{\BMS}(\Gamma_0\ba G)}.
$$
\end{thm}

Theorems \ref{mm1} and \ref{mix123}  are consequences of this theorem: if $\mu\in \mathcal P_{\mathsf{acc}}(\Gamma\ba G)$, then
$d\mu=\psi_2\, d\mathsf{m}^{\Haar}$ for some $\psi_2\in C_c(\Gamma\ba G)$ with $\int \psi_2 \, d\mathsf{m}^{\Haar}=1$.
Hence $$\int \psi_1 d\mu_t =\int \psi_1(ga_t) \psi_2(g)  \; d\mathsf{m}^{\Haar}(g).$$
Hence Theorem \ref{mm1} follows  if we put $\alpha(t)={t^{d/2} e^{(D-\delta)t}} $ and $c_\mu:=  \frac{\mathsf{m}^{\op{BR}_-}( \psi_2)}{(2\pi\sigma)^{d/2} \mathsf m^{\BMS}(\Gamma_0\ba G)}$.

For Theorem \ref{mix123}, note that when $\Gamma_0<G$ is cocompact, 
all the measures $\mathsf{m}^{\op{BR}_+}$, $\mathsf{m}^{\op{BR}_-}$ and $\mathsf{m}^{\op{BMS}}$ coincide with
$\mathsf{m}^{\Haar}$. Hence $c_\mu =\frac{1}{(2\pi\sigma)^{d/2} \mathsf m^{\Haar}(\Gamma_0\ba G)}$ depends only on $\Gamma$.

 \section{The $A$-ergodicity of generalized BMS measures}
 In this section, let $\Gamma$ be a non-elementary discrete subgroup of $G=\op{Isom}_+(\tilde X)$, and
  $\Omega\subset \Gamma \ba G/M$ denote the non-wandering set of the geodesic flow $\{a_t\}$.

\subsection{Generalized BMS-measures}
Let $\tilde F$ be a $\G$-invariant H\"{o}lder continuous function on $\T^1(\tilde X)$. Let $\chi:\Gamma\to \br$ be an additive character of $\Gamma$.

For all $x\ne y\in \tilde{X}$, we define $$\int_{x}^{y}\tilde{F}:=\int_{0}^{d(x,y)}\tilde{F}(va_t)dt$$ 
where $v$ is the unique  unit tangent vector based at $x$
such that $va_t$ is a vector based at $y$. The Gibbs cocycle for the potential $\tilde F$ is a map $C_{\tilde F}: \partial_\infty \tilde X \times \tilde X\times
\tilde X\to\br$ defined by
$$(\xi,x,y)\mapsto C_{\tilde F,\xi}(x,y)=\lim_{t\to +\infty} \int_y^{\xi_t} \tilde F -\int_x^{\xi_t} \tilde F$$
where $t\mapsto \xi_t$ is any geodesic ray toward the point $\xi$.

\begin{Def}\rm For $\sigma \in \br$, a twisted conformal density of dimension $\sigma$ for $(\Gamma, \tilde F, \chi)$
is a family of finite measures $\{\nu_x: x\in \tilde X\}$ on $\partial(\tilde X)$
such that for any $\gamma\in \G$, $x,y\in \tilde X$ and $\xi\in \partial(\tilde X)$,
$$\gamma_* \mu_x=e^{-\chi(\gamma)} \mu_{\gamma x}\quad \text{ and }\quad
\frac{d\mu_x}{d\mu_y} (\xi)=e^{-C_{\tilde F-\sigma,\xi}(x,y)}.$$\end{Def}

The twisted critical exponent $\delta_{\Gamma,\tilde F,\chi}$ of $(\Gamma, \tilde F, \chi)$ is given by
$$\limsup_{n\to +\infty} \frac{1}{n}\log \sum_{\gamma\in \Gamma, n-1<d(o,\gamma(o))\le n}
\exp \left(\chi(\gamma)+\int_{o}^{\gamma(o)}\tilde F\right).$$
When $\chi$ is trivial, we simply write it as $\delta_{\Gamma,\tilde F}$.
It can be seen that $\delta_{\Gamma,\tilde F}\le \delta_{\Gamma, \tilde F, \chi}$ for any character $\chi$ of $\Gamma$.

Suppose that $$\delta_{\Gamma, \tilde F, \chi}<\infty.$$ Then there exists a twisted conformal density of dimension $\delta_{\Gamma, \tilde F, \chi}$ for $(\Gamma, \tilde F, \chi)$ whose support is precisely the limit set of $\Gamma$; we call it a twisted Patterson-Sullivan density, or a twisted PS density for brevity. 
{Denote $\iota: \T^1(\tilde{X})\to \T^1(\tilde{X})$ to be the flip map, $v\mapsto -v$. It is shown in \cite[Proposition 11.8]{PPS} that $\delta_{\Gamma,\tilde{F},\chi}=\delta_{\Gamma, \tilde{F}\circ \iota,-\chi}$. We define the following generalized BMS measure:
\begin{defn}[Generalized BMS measures]
 Let $\{\mu_x:x\in \tilde{X}\}$ and $\{\mu_x^{\iota}:x\in \tilde{X}\}$ be twisted PS densities for $(\Gamma,\tilde{F},\chi)$ and $(\Gamma,\tilde{F}\circ \iota,-\chi)$ respectively. 
 Set $\delta_0=\delta_{\Gamma, \tilde{F}, \chi}$. A generalized BMS measure $\tilde{\mathsf{m}}=\tilde{\mathsf{m}}_{\Gamma}$ on $\T^1(\tilde X)=G/M$ associated to the pair
 $\{\mu_x:x\in \tilde{X}\}$ and $\{\mu_x^{\iota}:x\in \tilde{X}\}$
is defined by
\be\label{gems} d\tilde {\mathsf{m}} (u)= e^{C_{\tilde F-\delta_0, u^+}(o, u) +C_{\tilde F\circ \iota-\delta_0, u^-}(o, u)}d\mu_o(u^+)d\mu^{\iota}_o(u^-) ds\ee
using the Hopf parametrization of $\T^1(\tilde X)$. 
\end{defn}
By abuse
of notation, we use the notation $\tilde {\mathsf{m}}$ for the $M$-lift of $\tilde {\mathsf{m}}$ to $G$.

As $\chi$ and $-\chi$ cancels with each other, we can check that
the measure $\tilde{\mathsf{m}}_{\Gamma}$ is $\G$-invariant.
It induces  
\begin{itemize}
\item an $A$-invariant measure $\mathsf{m}^\dag_{\Gamma}$ on $\Gamma\ba G/M$ supported on $\Omega$ and
\item an $AM$-invariant measure $\mathsf{m}_{\Gamma}$ on  $\Gamma\ba G$. 
\end{itemize}

When there is no ambiguity about $\Gamma$, we will drop the subscript $\Gamma$ for simplicity.
When $\tilde{F}=0$ and $\chi$ is the trivial character, $\mathsf{m}$ is precisely equal to the BMS measure $\mathsf m^{\BMS}$
on $\Gamma\ba G$ defined in section 2.

\subsection{The $A$-ergodicity of generalized BMS measures}
The generalized Sullivan's dichotomy says that the dynamical system
 $(\Gamma\ba G/M, A,  \mathsf{m}^\dag)$ is either conservative and ergodic, or completely dissipative
 and non-ergodic \cite{PPS}.}

We will extend this dichotomy for the $A$-action on $(\Gamma\ba G, \mathsf m)$ using the density of the transitivity group
shown by Winter.

 \begin{defn}[Transitivity group]
 Fix $g\in \Omega$.
 We define the transitivity subgroup $ \mathcal{H}_{\G}(g)<AM$
 as follows: $ma \in \mathcal{H}_{\G}(g)$ if and only if
    there is a sequence $h_i\in N^-\cup N^+$, $i=1,\ldots, k$ and $\gamma \in \G $ such that
 \begin{align*}
  &\gamma gh_1h_2\ldots h_r\in \Omega\,\,\,\text{for all}\,\,\, 0\leq r\leq k,\,\,\,\text{and}\\
 &\gamma g h_1h_2\ldots h_k=gam. 
 \end{align*}
 \end{defn}

\begin{lem} \cite[Theorem 3.14]{Win} \label{win} Let $\G$ be Zariski dense.
The transitivity group $\mathcal H_\Gamma(g) $ is dense in $AM$ for any $g\in\Omega$.
\end{lem}

\begin{thm} \label{corr}Suppose that $\Gamma$ is Zariski dense.
Let $\mathsf{m}$ be a generalized BMS-measure on $\Gamma\ba G$ associated to $(\Gamma,\tilde F, \chi)$. 
If $(\Gamma\ba G/M, A,  \mathsf{m}^\dag)$ is conservative, then  $(\Gamma\ba G, A, \mathsf{m})$ is conservative and ergodic.

In particular, if $\Gamma$ is of divergence type, then
$(\Gamma\ba G, A, \mathsf{m}^{\BMS})$ is conservative and ergodic. \end{thm}

\begin{proof} Since the $A$-action on $(\Gamma\ba G/M, \mathsf{m}^\dag)$ is conservative
 and $\Gamma\ba G$
is a principal $M$-bundle over $\Gamma\ba G/M$, it follows that the $A$-action on $(\Gamma\ba G,\mathsf{m})$ is conservative as well:
 we can decompose  $\Gamma\ba G$  as $\Omega_C \cup \Omega_D$ where $\Omega_C$ and $\Omega_D$ are
 respectively the conservative and the dissipative parts
 of the $A$-action, that is,
$x \in \Omega_C$ iff $ xa_{t_i}$ comes back to a compact subset for some $t_i \to \infty$.
Note that $\Omega_C M$ is the conservative part for the geodesic flow, and must have the full $\mathsf{m}^{\BMS}$-measure
by the assumption. Since $a_t$ and $M$ commutes, $\Omega_C=\Omega_CM$, hence the claim follows.

We will now prove the $A$-ergodicity using the conservativity of the $A$-action, following Sullivan's argument 
which is based on Hopf's ratio ergodic theorem (see also \cite{Ro}, \cite{CI}).

Fix a positive Lipschitz map $\rho:\Gamma\ba G \to \br$ with $\int \rho d\, \mathsf{m} =1$;
fix $o\in \Omega$ and fix a positive continuous non-increasing function $r$ on $\br_{>0}$
which is affine on each $[n, n+1]$ and $r(n)=1/(2^n \mathsf{m}(B(o, n+1)))$.
Then $g\mapsto r( d(o, g(o)))$ is Lipschitz and belongs to $L^1(\mathsf{m})$; so by normalizing it, we get a function $\rho$ with desired properties (see \cite[Ch 5]{PPS} for more details).

By the conservativity,  $$\int_{0}^{\pm \infty} \rho (xa_t) dt =\pm \infty$$
for $\mathsf{m}$-almost all $x\in \G\ba G$.

 Now by the conservativity of  the $A$-action, we can apply the Hopf ratio theorem which  says  for any $\psi \in C_c(\Gamma \ba G)$,
$$\lim_{T\to {\pm \infty}}\frac{\int_{0}^T \psi (xa_t) dt}{\int_{0}^T \rho(xa_t) dt}$$
converges almost everywhere to an $L^1$-function $\tilde \psi^{\pm}$, and $\tilde \psi^+=\tilde \psi^-$ almost everywhere.
Moreover the $A$-action is ergodic if and only if $\tilde \psi^{\pm}$ is constant almost everywhere.

Using the uniform continuity of $\psi$ and $\rho$, we can first show that
 the limit functions $\tilde \psi^{\pm}$ coincide a.e. with an $N^+$ and $N^-$ invariant measurable function $\tilde \psi$. 
 Denote by $\psi^*$ the lift to $\tilde \psi$ to $G$.   
 Consider the Borel sigma algebra $\mathcal B(G)$ on $G$, and define
 subalgebras
 $\Sigma_{\pm}:=\{B\in \mathcal B(G): B=\Gamma B N^{\pm}\}$
 and $\tilde \Sigma:=\Sigma_-\wedge \Sigma_+$. That is, $B\in \tilde \Sigma$ if and only if there exist $B_{\pm}\in \Sigma_{\pm}$
 such that $\tilde {\mathsf{m}} (B\Delta B_{\pm})=0$.
 It is shown in \cite[Thm.4.3]{Win} that the density of the transitivity group and the ergodicity of $AM$-action on
 $(\Gamma\ba G, \tilde {\mathsf{m}})$ implies that
 $\tilde \Sigma$ is trivial (we note that the proof of \cite[Thm.4.3]{Win} does not require the finiteness of $\tilde {\mathsf{m}}$).
Both conditions are satisfied under our hypothesis by Lemma \ref{win} and the ergodicity of the system
$(\Gamma\ba G/M, A, \mathsf{m}^\dag)$ as remarked before.
 It now follows that  $\psi^*$  coincides with a constant function almost everywhere.
 This proves the $A$-ergodicity on $(\Gamma\ba G, \mathsf{m})$.
 
 The last part follows since $(\Gamma\ba G/M, A, \mathsf{m}^{\BMS})$ is conservative and ergodic when $\Gamma$ is of divergence
 type \cite{Su}.\end{proof}
 
\begin{cor}\label{twist} Let $\Gamma_0$ be a Zariski dense
 convex cocompact subgroup of $G$.
  Let $\tilde F$ be a $\G_0$-invariant H\"{o}lder continuous function on $\tilde X$
  and $\chi:\Gamma_0\to \br$ be a character. Suppose $\delta_{\Gamma_0, \tilde F, \chi}<\infty$.

\begin{enumerate} 
\item If  $\mathsf{m}_{\Gamma_0}$ is a  generalized BMS-measure on $\Gamma_0\ba G$ 
 associated to $(\Gamma_0,\tilde F, \chi)$, then $(\Gamma_0\ba G, A, \mathsf{m}_{\Gamma_0})$ is ergodic and conservative.

\item Let $\Gamma <\Gamma_0$ be a normal subgroup with $\Gamma\ba \Gamma_0=\z^d$, and $\chi=0$ be the trivial character.
Let  $\mathsf{m}_{\Gamma}$ be the measure on $\Gamma\backslash G$ induced by the generalized BMS-measure $\tilde{\mathsf{m}}_{\Gamma_0}$ on $G$  associated to  $(\Gamma_0,\tilde F, 0)$. Then
 $(\Gamma \ba G, A, \mathsf{m}_{\Gamma})$ is ergodic and conservative if and only if $d\le 2$.
\end{enumerate}
 \end{cor}
 \begin{proof} 
 Since $\Gamma_0$ is convex cocompact, $\mathsf{m}_{\Gamma_0}^\dag$ on $\Gamma_0\ba G/M$ is compactly supported and hence
 conservative. Therefore Theorem \ref{corr} shows the claim (1).
 
 For (2), denote by $\mathsf{m}_{\Gamma}^\dag$ the measure on $\Gamma \ba G/M$ induced by $\tilde{\mathsf{m}}_{\Gamma_0}$.
 The $A$-action on $(\Gamma\ba G/M, \mathsf{m}^\dag_{\Gamma})$ is ergodic
 and conservative if and only if $d=1,2$ (\cite{Re},\cite{Yu}). Hence (2) again follows from Theorem \ref{corr}.
  \end{proof}

Theorem \ref{Ree} is a special case of Corolloary \ref{twist}(2).

\section{Babillot-Ledrappier measures and measure classification}

 Let $G$ be a connected simple linear Lie group of rank one. Let $\Gamma_0$ be a discrete subgroup of $G$.
 For a subset $S$ of $G$ and $\e>0$, let $S_\e$ denote the intersection of $S$ and the $\e$-ball of $e$ in $G$. We also use the notation $S_{O(\e)}$ to denote the set $S_{C\e}$ for some constant $C>0$ depending only on $G$ and $\Gamma_0$.
\subsection{Closing lemma.}
Let  $$\mathcal B( \e):= (N_\e^+N^-\cap N_\e^-N^+AM) M_\e A_\e .$$

\begin{lem}[Closing lemma] \cite[Lemma 3.1]{MMO}\label{closing}
There exist $T_0>1$, and $\e_0>0$  depending only on $G$ for which the following holds: 
 if $$g_0\mathcal B(\e) a_T m \cap\gamma g_0 \mathcal B(\e)\ne\emptyset$$ for some $T>T_0$, $0<\e<\e_0$,
 $m\in M$, $g_0\in G$ and $\gamma \in G$, then there exists $g\in g_0\mathcal B(2\e)$ such that
 $$\gamma = g a_0 m_0g^{-1}$$ where
 $a_0 \in A$ and $m_0 \in M$ satisfy
 $a_0\in a_T A_{O( \e)}$ and $m_0 \in m M_{O( \e)}$.
\end{lem}

\subsection{Generalized length spectrum}
In the rest of this section, let $\Gamma$ be a normal subgroup of a Zariski dense convex cocompact subgroup $\Gamma_0$
 of $G$.
For $\gamma \in \Gamma_0$, we have 
$\gamma= g a m g^{-1}$ for some $g\in G$, $am\in AM$.
The element $g$ is determined in $G/M$, and $m$ is determined only up to conjugation in $M$.
Fix a section $\mathcal L=N^+AN^-$ for $G/M$.
Let 
$$\Gamma_0^\star=\{\gamma\in \Gamma_0: \gamma= g am g^{-1} \text{ for some $g\in \mathcal L$}\}.$$
For $\gamma\in \Gamma_0^\star$,
there are unique $g\in \mathcal L, l(\gamma)\in \mathbb{R}_{>0}$ and $m_\gamma \in M$ such that  
$$\gamma= g a_{l(\gamma)} m_\gamma g^{-1}.$$

For $\gamma \in \Gamma_0$, we write $f(\gamma) \in \Gamma\ba \Gamma_0$
for its image under the projection map $\Gamma_0\to \Gamma\ba \Gamma_0$.
Hence we can write $$\Gamma  f(\gamma) g=\Gamma g a_{l(\gamma)} m_\gamma.$$

\begin{Def} The generalized length spectrum $\mathcal {GL} (\Gamma_0, \Gamma)$ of  $\Gamma_0$ relative to
$\Gamma$ is defined as
$$\mathcal {GL} (\Gamma_0, \Gamma):= \{(f(\gamma), a_{l(\gamma)}, m_\gamma)\in \Gamma\ba \Gamma_0\times A\times M:\gamma\in \Gamma_0^\star\}.$$
\end{Def}

\begin{prop}\label{dens} Suppose that $\Gamma $ is a normal subgroup of $\Gamma_0$
and that $(\Gamma\ba G, a_t)$ satisfies the topological mixing property:
for any two open subsets $U, V\subset \Gamma\ba G$,
$$Ua_t\cap V\ne\emptyset$$ for all sufficiently large $t>1$.
Then the subgroup  generated by $\mathcal {GL} (\Gamma_0, \Gamma)$ is dense in $\Gamma\ba
\Gamma_0 \times A\times M$.
\end{prop}
\begin{proof}
Consider $(\xi, a_T, m)\in \Gamma\ba
\Gamma_0\times A\times M$.
%Let $$\mathcal B(\e):=(N_\e^+N^-\cap N_\e^-N^+AM) M_\e A_\e .$$
The assumption on the topological mixing property implies that for any small $\e>0$, there is $T_\e>0$ such
that  $$\Gamma   \mathcal B(\e) m a_T  \cap \Gamma  \xi \mathcal B(\e) \ne \emptyset$$
for all $T>T_\e$. That is, for all $T>T_\e$, there exist $g_1, g_2\in  \mathcal B(\e) $ such that
$  g_1 m a_T= \xi \gamma g_2$ for some $\gamma \in \Gamma$.
Set $\gamma_0:=\xi \gamma\in \Gamma_0$; so $f(\gamma_0)=\xi$.
The closing lemma \ref{closing} implies that there exists $g\in \mathcal B(2\e)$ such
that $\gamma_0 = g a_{T_0}m_{\gamma_0} g^{-1}$
with $a_{T_0-T}\in A_{O(\e)}$ and $m_{\gamma_0}=m M_{O(\e)}$.

Hence for any $\e>0$ and for any $(\xi, a_T, m)\in \Gamma\ba \Gamma_0\times A\times M$ for $T$ sufficiently large,
we can find $\gamma\in \Gamma_0^\star$ such that $(f(\gamma), a_{l(\gamma)}, m_\gamma)$
is within an $O(\e)$-neighborhood of $(\xi, a_T, m)$. This proves the claim.
\end{proof}

We deduce the following from Proposition \ref{dens} and Theorem \ref{mix2}: 
\begin{cor}\label{cowbell} Suppose that $\Gamma $ is a co-abelian subgroup of $\Gamma_0$. Then
the group  generated by $\mathcal {GL} (\Gamma_0, \Gamma)$ is dense in $\Gamma\ba
\Gamma_0 \times A\times M$.

\end{cor}

\subsection{The $N^-$-ergodicity of generalized BMS measures}
  Let $\tilde F$ be a $\G_0$-invariant H\"{o}lder continuous function on $\T^1(\tilde X)$
  and $\chi:\Gamma_0\to \br$ be a character. Suppose $\delta_{\Gamma_0, \tilde F, \chi}<\infty$.
Let  $\tilde{\mathsf m}_{\Gamma_0}$ be a  generalized BMS-measure on $G$  associated to $(\Gamma,\tilde F,\chi)$.
 By Corollary \ref{twist}, the induced measure $\mathsf m_{\Gamma_0}$  on $\Gamma_0\backslash G$ is $A$-ergodic.
We denote by $$\mathsf m=\mathsf{m}_\Gamma$$ the $AM$-invariant measure on $\Gamma\ba G$ induced by $\tilde{\mathsf m}_{\Gamma_0}$.

Note that any essentially $N^-$-invariant measurable function
$\psi$ in $(\Gamma \ba G, \mathsf m)$ is almost everywhere
equal to a measurable $N^-$-invariant function \cite{Zi}.
 
 Define $H=H(\mathsf m)$ to be the set
of $ (\xi, a, m)\in \Gamma\ba \Gamma_0 \times A\times M$ such that
for all $N^-$-invariant measurable function $\psi$ in $L^{\infty}(\Gamma\ba G, \mathsf m)$, 
$$\psi(x )=\psi(\xi^{-1} x am)$$ for $\mathsf m$-almost all $x\in \Gamma\ba G$.

It is easy to check that $H$ is a closed subgroup.

\begin{prop} \label{mainpp} 
If $\Gamma<\Gamma_0$ is co-abelian, then
$$H(\mathsf m) =\Gamma\ba \Gamma_0 \times A\times M.$$
\end{prop}
\begin{proof}{ By Corollary \ref{cowbell},
it suffices to show that
$$\mathcal {GL} (\Gamma_0,\Gamma)\subset H .$$
Fix $(f(\gamma), a_{l(\gamma)}, m_\gamma)\in \mathcal{GL}(\Gamma_0,\Gamma)$, that is,
for some unique $g\in \mathcal L$, $$f(\gamma) \Gamma g = \Gamma g a_{l(\gamma)} m_\gamma.$$

Let $\pi:\Gamma\ba G\to \Gamma_0\ba G$ be the canonical projection.
Let $[g]=\Gamma g$ and  let $B(\pi [g], \e)$ denote the $\e$-ball around $\pi [g]$.
Let $p>0$ be such that for any $x\in B(\pi [g], \e/p)$ and any $0\leq t< 2\eg$, we have
$$d(xa_t,\pi[g]a_t)<\epsilon/2.$$
Fix $\e>0$, and set $\mathfrak S_k=\mathfrak S_k(\e) \subset \Gamma\ba G$ be the set of $x$'s such that
$\pi(x)a_t\in B(\pi [g], \e/p)$ 
for some $t\in [k\eg , (k+1) \eg)$.

Since $\mathsf m_0 $ is $A$-ergodic, 
$$\mathsf m( \Gamma\ba G- \cup_k \mathfrak S_k)=0 .$$
 Let $x\in \mathfrak S_k$. By replacing $g$ by $ga_t$ for some $0\leq t<\eg$, we have
$$d(xa_{k\eg}, \xi [g]) \le \e/2\text{ and }d(xa_{k\eg+\eg }, \xi [g]a_{\eg} ) \le \e/2 .$$

Since $ \xi [g]a_{\eg}= f(\gamma) \xi [g] m_\gamma^{-1}$ and the metric $d$ is right $M$-invariant and left $G$-invariant,
$$d(xa_{k\eg+\eg }, \xi [g]a_{\eg} )= d( f(\gamma)^{-1} xa_{k\eg+\eg }m_\gamma ,\xi [g] ) .$$
Hence
$$d( f(\gamma)^{-1}x a_{(k+1)\eg} m_\gamma, x a_{k\eg}) \le \e .$$

Since the product map $N^-\times A\times M\times N^+\to G$ is a diffeomorphism onto
an open neighborhood of $e$ in $G$,
there exist an element  $n^- am n^+\in N^-_{O(\e)} A_{O(\e )}M_{O(\e)} N_{O(\e)}^+$ such that
 $$ x a_{k\eg} = x f(\gamma)^{-1} a_{(k+1)\eg} m_\gamma n^- am n^+ .$$

Set $$ x^*:= x (a_{(k+1)\eg} m_\gamma)  n^- (a_{(k+1)\eg} m_\gamma)^{-1} \in xN^-$$
and $$T_k(x):= f(\gamma)^{-1} x^* a_{\eg}m_\gamma am %(a_{k\eg}n^+ a_{k\eg}^{-1})
.$$
Note
$$x=T_k(x) a_{k\eg} n^+ a_{k \eg}^{-1}\in T_k(x)N^+_{O(e^{-k}\e)}.$$

Now suppose that
$(f(\gamma), a_\gamma, m_\gamma)\notin H$, that is, there exists an $N^-$-invariant measurable function
$\psi\in L^{\infty}(\mathsf m)$, a compact set $W$ with $\mathsf m(W)>0$, and $\e_0>0$ such that
$\psi(xh)=\psi(x)$ for all $x\in W$ and $h\in N^-$ and
$$\psi(x)>\psi(f(\gamma)^{-1} x a_\gamma m_\gamma)+\e_0$$
for all $x\in W$.
Since $x^*\in xN^-$, we have for all $x\in W$,
$$ \psi(x)> \psi(f(\gamma)^{-1} x^* a_\gamma m_\gamma) +\e_0 .$$
If we consider a sequence of non-negative continuous functions
$\eta_{\delta } $ with integral one and which is supported
in $(AM)_\delta$, then as $\delta\to 0$,
 the convolution $\psi*\eta_\delta$ converges to $\psi$ almost everywhere.
 By replacing $\psi$ with $\psi*\eta_\delta$ for small $\delta$, 
we may assume that $\psi$ is continuous for the $AM$-action. Moreover, using Luzin's theorem,
we may assume that $\psi$ is uniformly continuous on $W$ by replacing $W$ by a smaller subset if necessary.

%We apply the above construction for $\e=\e_0/4$. 
Since $$T_k(x) = f(\gamma)^{-1} x^* a_{\eg} m_\gamma am \to x $$ as $k\to \infty$ and $am\in (AM)_{O(\e)}$.
we will get a contradiction
if we can find $x\in W$ with $ T_{k}(x)\in W$ for an arbitrarily large $k$.

Hence it remains to prove the following claim: 
\begin{enumerate}
\item $\lim_k \mathsf m(\{x\in W\cap \mathfrak S_k: T_k(x)\notin W\})= 0$;
\item $\limsup_k \mathsf m (W\cap \mathfrak S_k) >0$.
\end{enumerate}

Note that $T_k$ maps $x$ to $f(\gamma)^{-1} x a_{(k+1)\eg}m_\gamma   a_{-k\eg} am$.
Since $\mathsf m$ is $AM$-invariant,  $\Gamma\ba \Gamma_0$-invariant, and $n^-\in G_\e$,
there exists $\kappa=\kappa (\e)>1$ such that $\mathsf m (T_k^{-1} (Q))\le \kappa \mathsf m (Q)$ for any Borel set $Q
\subset \Gamma\ba G$ and any $k\ge 1$

Since for $x\in \mathfrak S_k$, $d(x, T_k(x))\to 0$, we have
$$\mathsf m(\{x\in W\cap \mathfrak S_k: T_k(x)\notin W\})\le\mathsf m (T_k^{-1} ( WG_\eta -W))\le \kappa \cdot
\mathsf m ( WG_\eta -W)$$
for small small $\eta>0$ which goes to $0$ as $k\to \infty$.

Now as $W$ is a compact subset, we have $\mathsf m(WG_\eta -W)\to 0$ as $\eta\to 0$, and hence
the first claim follows.

For the second claim, suppose the claim fails. Note that for any $t>0$, there exists $k\in\mathbb{N}$ such that $\pi^{-1}(B(\pi[g],\epsilon))a_{-t}\subset \mathfrak S_k$.
Hence we have $$\limsup_t\mathsf m(W\cap \pi^{-1}( B(\pi [g], \e) a_{-t}) = 0 .$$ 
This would mean that
$$\limsup_s\frac{1}{s} \int_0^s \int_W 1_{B(\pi [g],\e)} (\pi(x) a_t) d\mathsf m dt = 0.$$
However, by the Fubini theorem, the Birkhoff ergodic theorem (as $\mathsf m_0$ is $A$-ergodic) and the Lebesgue dominated covergence theorem, 
we deduce
\begin{align*} \frac{1}{s} \int_0^s \int_W 1_{B(\pi [g],\e)} (\pi(x) a_t) d\mathsf m dt
=&\int_W\frac{1}{s}\int_0^s  1_{B(\pi [g],\e)} (\pi(x) a_t) dt d\mathsf m
\\ \to & \mathsf m_0 (B(\pi [g],\e)) \mathsf m(W) .\end{align*}
Therefore this proves the claim.}
\end{proof}

 \begin{thm}\label{ner} If $\Gamma\ba \Gamma_0$ is abelian,
 then the $N^-$-action on $(\Gamma\ba G, \mathsf m)$ is ergodic, i.e. any $N^-$-invariant measurable function on
 $(\Gamma\ba G,\mathsf m)$ is a constant.
\end{thm}
\begin{proof}
Since $H=H(\mathsf m)=\Gamma\ba \Gamma_0\times A\times M$ by Proposition \ref{mainpp},
any $N^-$-invariant measurable function $\psi$ on $(\Gamma\ba G,\mathsf m) $
gives rise an $AMN^-$-invariant function $\tilde \psi$ on $(\Gamma_0\ba G, \mathsf m_{\Gamma_0})$.

Since $\mathsf m_{\Gamma_0}$ is $A$-ergodic, $\tilde \psi$ is constant $\mathsf m_{\Gamma_0}$ almost everywhere.
It follows that $\psi$ is constant  $\mathsf m$-almost everywhere. Therefore the claim is proved.
\end{proof}

\subsection{Babillot-Ledrappier measures}
Since $\Gamma_0$ is convex cocompact, 
we have $$\delta_{\Gamma_0, \chi}<\infty $$
for any character $\chi$ of $\Gamma_0$ (see \cite[Proof of Prop. 11.8]{PPS}).

Suppose that $\Gamma$ is co-abelian in $\Gamma_0$, and
consider the space $(\Gamma\ba \Gamma_0)^*$ of characters of $\Gamma_0$ which vanish on $\Gamma$.
Fix a character $\chi\in(\Gamma\ba \Gamma_0)^*$.
There exists a unique twisted PS density $\{\mu_{\chi,x}:x\in \tilde{X}\}$ for $(\Gamma_0,0,\chi)$ supported on
$\Lambda(\Gamma_0)$ \cite[Corollary 11.13]{PPS}. 

\begin{Def}\label{bab}  Define the following measure on $\T^1(\tilde{X})$ using the Hopf parametrization:
\begin{equation*}
d\tilde{\mathsf{m}}_{\chi}(u)=e^{\delta_{\Gamma_0,\chi}\beta_{u^+}(o,u)+D\beta_{u^-}(o,u)}d\mu_{\chi,o}(u^+)d\mathsf m_{o}(u^-)ds,
\end{equation*}
where $d\mathsf m_o$ is the Lebesgue density on $\partial (\tilde{X})$. 
\end{Def} 

One can check that the measure $\tilde{\mathsf{m}}_{\chi}$ satisfies the following properties:
\begin{enumerate}
\item identifying $\T^1(\tilde{X})$ with $G/M$, $\tilde{\mathsf{m}}_{\chi}$ is $N^-$-invariant (i.e., Lebesgue measures on each $N^-$
leaf) and quasi $A$-invariant;

\item as $\chi$ vanishes on $\Gamma$, $\tilde{\mathsf{m}}_{\chi}$ is $\Gamma$-invariant;

\item for any $\gamma\in (\Gamma\ba \Gamma_0)^*$, $\gamma_*\tilde{\mathsf{m}}_{\chi}=e^{-\chi(\gamma)}\tilde{\mathsf{m}}_{\chi}$.
\end{enumerate}
Denote by $\mathsf{m}_{\chi}$ the $MN^-$-invariant measure on $\Gamma\backslash G$ induced by $\tilde{\mathsf{m}}_{\chi}$; we call
it a Babillot-Ledrappier measure. When $\chi=0$ is the trivial character, $\mathsf m_0$ coincides with the Burger-Roblin mesure $m^{\BR_-}$ up to a constant multiple.
 
 Consider the generalized BMS measure ${\mathsf{m}}_{\Gamma_0}$ on $\Gamma\ba G$ associated to the pair
 $\{\mu_{\chi, x}\}$ and $\{\mu_{-\chi, x}\}$. Then  ${\mathsf{m}}_{\Gamma_0}$ and $\mathsf m_\chi$ have the same
transverse measure $$e^{\delta_{\Gamma_0,\chi}\beta_{u^+}(o,u)}d\mu_{\chi,o}(u^+)ds dm.$$
Hence the $N^-$-ergodicity of $\mathsf m_\chi$ is equivalent to the $N^-$-ergodicity of
$\mathsf m_\Gamma$. 
 
 Using Theorem \ref{ner}, we obtain:
\begin{Thm}\label{eee} For each $\chi\in (\Gamma\ba \Gamma_0)^*$,
the measure $\mathsf m_\chi$ is $N^-$-ergodic.
\end{Thm}
Moreover
when $\Gamma_0$ is cocompact, Sarig  \cite{Sa} and Ledrappier \cite[Corollary 1.4]{Le} showed that
any $N^-M$ invariant ergodic measure on $\Gamma\ba G$ is one of $\mathsf m_\chi$'s,
by showing that such a measure should be $A$-quasi-invariant and then using
the classification of Babillot on $N^-M$-invariant and $A$-quasi-invariant measures \cite{Ba1}.

\begin{thm} Suppose that $\Gamma_0<G$ is cocompact and that $\Gamma<\Gamma_0$ is co-abelian.
Then any $N^-$-invariant ergodic measure is proportional to $\mathsf m_\chi$ for some
$\chi\in (\Gamma\ba \Gamma_0)^*$.
 \end{thm}
 \begin{proof}
 If $\nu$ is an  ergodic $N^-$-invariant measure on $\Gamma\ba G$, then
 the average $\nu^\clubsuit:=\int_M m_*\nu dm$ is an ergodic $N^-M$ invariant measure, and hence
 is proportional to $\mathsf m_\chi$ for some
$\chi\in (\Gamma\ba \Gamma_0)^*$ by the afore-mentioned results (\cite{Sa}, \cite{Le}).
 As $\mathsf m_\chi$  is $N^-$-ergodic by Theorem \ref{eee}, it follows that
  for almost all $m\in M$, $m_*\nu$ is proportional to $\mathsf m_\chi$. Pick one such $m\in M$.
 Since $\mathsf m_\chi$ is $M$-invariant, we have $\nu$ is proportional to 
 $(m^{-1})_*\mathsf m_\chi=\mathsf m_\chi$. \end{proof}

 \section{Applications to counting and equidistribution problems}
 As before, let $\Gamma$ be a normal subgroup of a Zariski dense convex cocompact subgroup $\Gamma_0$
 of $G$ with $\Gamma\ba \Gamma_0\simeq \z^d$ for some $d\ge 0$. We
 describe counting and equidistribution results for $\Gamma$ orbits which can be deduced
 from Theorems \ref{m1} and \ref{mix22}. The deduction process is well-understood (see \cite{OS}, \cite{MO}, and  \cite{MMO}).
 
 \subsection{ Equidistribution of translates.}
 Let $H$ be an expanding
  horospherical subgroup or a symmetric subgroup of $G$.
 Recall  the definition the PS measure on $gH/(H\cap M)\subset G/M=\T^1(\tilde X)$: 
 $$d\tilde \mu_{gH}^{\PS}(gh)=e^{\delta\beta_{(gh)^+} }d\nu_o((gh)^+).$$ We denote by
 $\tilde \mu_{gH}$ the $H\cap M$ invariant lift to the orbit $gH$.

 \begin{thm}\label{eq}  Let $H$ be the expanding
  horospherical subgroup or a symmetric subgroup of $G$, and let 
 $x=\Gamma g \in \Gamma\ba G$. Let  $U\subset H$ be a small open
  subset such that $u\mapsto xu$ is injective on $U$.
 Suppose that $\tilde \mu_H^{\PS}(\partial(gU))=0$.
 Then we have
 $$\lim_{t\to \infty} t^{d/2} \int_{U} \psi(\Gamma g ua_t) d\tilde \mu_H^{\PS}(gu) 
 = \frac{\tilde \mu_H^{\PS}(gU)}{(2\pi\sigma)^{d/2} } \mathsf{m}^{\BMS}(\psi)$$
and
$$\lim_{t\to \infty} t^{d/2} e^{(D-\delta)t} \int_{U} \psi(xua_t) du
 = \frac{\tilde \mu_H^{\PS}(gU)}{(2\pi\sigma)^{d/2}} \mathsf{m}^{\BR_+}(\psi).$$
 \end{thm}

The assumption on $H$ being either symmetric or horospherical
ensures the wave front property of \cite{EM} which can be used to establish, as $t\to \infty$,
\be \label{tt} \int_{U}\psi(\Gamma g u a_t) d\tilde \mu_H^{\PS}(gu) 
\approx \la a_t \psi, \rho_{U, \e}\ra_{m^{\BMS}}
\text{ and } \int_{U}\psi(xu a_t) du
\approx \la a_t \psi, \rho_{U, \e}\ra 
 \ee
where
$\rho_{U, \e} \in C_c(\G\ba G)$ is an $\e$-approximation of $U$, and $\approx$ means that the ratio of the two terms is of size
$1+O(\e)$.
Therefore the estimates on the matrix coefficients
in Theorems \ref{m1} and \ref{mix2} can be used to establish Theorem \ref{eq}. We refer to
\cite[Sec. 3]{OS} and \cite{MO} for more details.

Suppose that $\G_0\ba\G_0 H$ is closed. This implies that $\G\ba \G H$
is closed too. Then $\tilde \mu_H^{\PS}$ induces  locally finite Borel measures  
$\mu_{\G_0, H}^{\PS}$ and $\mu_{\G, H}^{\PS}$  on $\G_0\ba \G_0 H$ and $\G\ba \G H$
respectively.

As $\Gamma_0$ is convex cocompact,  $\mu_{\G_0, H}^{\PS}$ is compactly supported \cite[Theorem 6.3]{OS}.
If $H\cap \G<H\cap \Gamma_0$ is of finite index, the measure $\mu_{\G, H}^{\PS}$ is also compactly supported.
In this case, by applying Theorem \ref{eq} to the support of $\mu_{\G, H}^{\PS}$, we get the following:
\begin{thm} \label{op} Suppose that $\G_0\ba\G_0 H$ is closed and that $[H\cap \G_0: H\cap \G]<\infty$. Then 
\be \label{cou} \lim_{t\to \infty} t^{d/2} e^{(D-\delta)t} \int_{H} \psi([e]ha_t) dh
 = \frac{|\mu_{\Gamma, H}^{\PS}|}{(2\pi\sigma)^{d/2}} \mathsf{m}^{\BR_+}(\psi) .\ee
\end{thm}
In the case when $H=K$ and $\G_0$ is torsion free, we have
$K\cap \G=K\cap \Gamma_0=\{e\}$, and Theorem \ref{op} provides the equidistribution
of Riemannian spheres $S_t$ with respect to $(\Gamma\ba G, \mathsf{m}^{\BR_+})$ as $t\to \infty$. 

 \subsection{Distribution of a discrete $\G$-orbit on $H\ba G$.}
Let
$H$ be either a symmetric subgroup, a horospherical subgroup or the trivial
subgroup of $G$. 
 Consider the homogeneous space $H\ba G$
and suppose $[e]\Gamma_0$ is discrete in $H\ba G$. 
Recall that $v_o\in \T^1(\tilde X)$ is a vector whose stabilizer is $M$.
  \begin{Def}\label{MD}
 \begin{enumerate}
 \item Define a Borel measure $\mathcal M=\mathcal M_{\Gamma,G}$ on $G$ as follows: for $\psi\in C_c(G)$,
 $$\mathcal M(\psi):=\frac{1}{(2\pi\sigma)^{d/2} } \int_{k_1a_tk_2\in KA^+K} \psi(k_1a_tk_2) 
 e^{\delta t } t^{d/2} d\nu_o(k_1 v_o^+) dt d\nu_o(k_2^{-1} v_o^-) .$$
 \item For $H$ symmetric or horospherical,
  we have either $G=HA^+K$ or $G=HA^+K\cup HA^-K$.  
 Define a Borel measure $\mathcal M=\mathcal M_{\Gamma, H\ba G}$ on $H\ba G$ as follows: if $G=HA^+K$
 and $\psi\in C_c(H\ba G)$,
 $$\mathcal M(\psi):=\frac{ |\mu_{\Gamma, H}^{\PS}|}{(2\pi\sigma)^{d/2} } \int_{a_tk\in A^+K} e^{\delta t }\cdot t^{d/2} \psi([e]a_t k)
  dt d\nu_o(k^{-1} v_o^-).$$
  If $G=HA^+K\cup HA^-K$ and  $\psi\in C_c(H\ba G)$,
  $$\mathcal M(\psi):=\sum
  \frac{|\mu_{\Gamma,H,\pm}^{\PS}|}{(2\pi\sigma)^{d/2} } \int_{a_{\pm t}k\in A^{\pm} K} e^{\delta t}\cdot  t^{d/2} \psi([e]a_t k)
  dt d\nu_o(k^{-1} v_o^{\mp}).$$
  \end{enumerate}
 
 \end{Def}

For a compact subset $B\subset H\ba G$,
define $B_{ \e}^+=BU_\e$ and $B_{\e}^-=\cap_{u\in U_\e} Bu$ if $H\ne \{e\}$. 
For $H=\{e\}$,
 set $B_{ \e}^+= U_\e B U_\e$ and  $B_{\e}^-=\cap_{u_1, u_2 \in U_\e} u_1 Bu_2$. 
 
 \begin{Def}\label{MDW}
 Let $B_T$ be a family of compact subsets in $H\ba G$.
 We say $B_T$ is well-rounded with respect to $\mathcal M=\mathcal M_{\Gamma,H\ba G}$ if 
 \begin{enumerate}
 \item $\mathcal M(B_T)\to \infty$ as $T\to \infty$;
 \item
 $\limsup_{\e\to 0}\limsup_{T\to 0} \frac{\mathcal M({B_{T,\e}^+ -B_{T,\e}^-)}}{\mathcal M({B_T})}=0.$
 \end{enumerate}
 \end{Def}

 The following theorem  can be proved
 using the equidistribution theorem \ref{op} for $H$ symmetric or horospherical, and
  Theorem \ref{mix2} for $H$ trivial (see \cite[Sec. 6]{Oh}).
 
 \begin{thm} Suppose that $[e]\Gamma_0$ is discrete and that $[H\cap \Gamma_0:H\cap \Gamma]<\infty$.
 If $\{B_T\}$ is a  sequence of compact subsets in $H\ba G$ which are well-rounded with respect to $\mathcal M_{\Gamma,H\ba G}$,
 then as $T\to \infty$,
 $$\# [e]\Gamma \cap B_T \sim \mathcal M_{\Gamma,H\ba G}(B_T) .$$
 \end{thm}
 This was shown in \cite{OS} and \cite{MO} for $d=0$.
 
  %\noindent{\bf }
\subsection{Distribution of circles} We also state the following result on the asymptotic distribution of a $\Gamma$-orbit of a circle
whose proof can be obtained in the same way as in \cite{Oh-Shah2} using Theorem \ref{mix2} (see also \cite{KO}, \cite{LO}, \cite{Pan}).
  \begin{thm} \label{sk2} Let $G=\op{PSL}_2(\mathbb C)$. 
  Let $C_0$ be a circle in the complex plane $\mathbb C$ such
  that $\P:=\Gamma_0 (C_0)$ is discrete in the space of circles in $\mathbb C$. Suppose
  that $\op{Stab}_\Gamma (C_0)=\op{Stab}_{\Gamma_0} (C_0)$.
 There exists $c>0$ such that for any compact subset $E\subset \mathbb C$ whose boundary is rectifiable,
 we have 
 $$\#\{C\in \P: C\cap E\ne\emptyset, \op{curv}(C)\le T\} \sim c \frac{ T^{\delta}}{(\log T)^{d/2}} \op{H}^\delta(E) $$
 where $\op{H}^\delta$ is the $\delta$-dimensional Hausdorff measure of the limit set of $\Gamma$ (with respect to the Euclidean metric).
 \end{thm}

\subsection{ Joint equidistribution of closed geodesics and holonomies}
 Recall $X=\Gamma\ba \tilde X$. A primitive closed geodesic $C$ in $\T^1(X)$ is a compact set of the form
 $$\G\ba \G gAM/M=\G\ba \G g A(v_o)$$
 for some $g\in G$. The length of a closed geodesic  $C=\G\ba \G gAM/M$ is given as the co-volume of 
$AM\cap g^{-1}\G g $ in $AM$. 

 For each closed geodesic $C=\G\ba \G gAM/M$, we  write $\ell(C)$ for the length of $C$, 
 and denote by  $h_C$
 the unique $M$-conjugacy class associated to the holonomy class of $C$: $\G g m a_{\ell(C)} =\Gamma g m h_C$.
We also write
  $\mathcal{L}_C$ for the length measure on  $C$

Let $M^{\textsc{c}}$ denote the space of conjugacy classes of $M$.
For $T>0$, define $$\mathcal G_\G(T):=\{C: C\text{ is a closed geodesic in $\T^1(X)$},\;\; \ell(C)\le T\}.$$

For each $T>0$, we define
 the measure $\mu_T$ on the product space $(\G\ba G/M) \times M^{\textsc{c}}$:
  for $\psi \in C(\Gamma\ba G/M)$ and any class function $\xi\in C(M)$,
$$\mu_T(\psi \otimes \xi)=\sum_{C\in  \mathcal G_\G(T)} \mathcal{L}_C (\psi ) \xi(h_C).$$

We also define a measure  $\eta_T$  by
$$\eta_T(\psi \otimes \xi)=\sum_{C\in  \mathcal G_\G(T)} \mathcal{D}_C (\psi) \xi(h_C),$$
 where 
$ \mathcal{D}_C (\psi )=\ell(C)^{-1} \mathcal{L}_C (\psi )$.
 \medskip

 Given Theorem \ref{mix2},
 the following can be deduced in the same way as the proof of Theorem 5.1 in \cite{MMO}.
\begin{thm}\label{eqt} For any bounded $\psi \in C(\G\ba G/M)$ and a class function $\xi \in C(M)$,
we have, as $T\to \infty$, \begin{equation*}\label{mue2}
\mu_{T }( \psi \otimes \xi)
\sim  \frac{e^{\delta T} }{(2\pi\sigma)^{2/d} \delta T^{d/2}} \cdot  \mathsf{m}^{\BMS}(\psi)  \cdot \int_M \xi dm  \end{equation*}
and \begin{equation*}\label{mue3}
 \eta_{T }( \psi \otimes \xi)
\sim     \frac{ e^{\delta T}  }{(2\pi\sigma)^{2/d} \delta T^{d/2+1}  } 
 \cdot  \mathsf{m}^{\BMS}(\psi)  \cdot \int_M \xi dm.\end{equation*}\end{thm}
 
 Theorem \ref{mp1} is an easy consequence of this.

\end{document}